\documentclass{amsart}

\usepackage{fullpage}
\usepackage{faktor}

\usepackage{amsmath}
\usepackage{amssymb}
\usepackage{amsfonts}
\usepackage{mathtools}
\usepackage{tikz}
\usepackage{amsthm}
\usepackage{enumerate}
\usepackage{pgf,tikz,pgfplots}
\pgfplotsset{compat=1.15}
\usepackage{mathrsfs}
\usetikzlibrary{arrows}
\usepackage{bbm}
\usepackage{ytableau}
\usepackage{verbatim}
\usepackage{tikz-cd}

\newtheorem{mquestion}{Main Question}
\newtheorem*{acknowledgement}{Acknowledgement}
\newtheorem{oquestion}{Open Question}
\newtheorem{theorem}{Theorem}[section]
\newtheorem{lemma}[theorem]{Lemma}
\newtheorem{remark}[theorem]{Remark}
\newtheorem{fact}[theorem]{Fact}
\newtheorem{corollary}[theorem]{Corollary}
\newtheorem{example}[theorem]{Example}
\newtheorem{definition}[theorem]{Definition}
\newtheorem{proposition}[theorem]{Proposition}
\newtheorem*{conjecture}{Conjecture}
\newtheorem*{recall}{Recall}
\newtheorem*{motivation}{Motivation}

\theoremstyle{plain}
\newtheorem{case}{Case}
\newcommand{\pibar}{\overline{\pi}}

\begin{document}
\pagestyle{plain}

\title{Three Examples of Quasisymmetric Compatible $\mathfrak{S}_n$-modules}
\date{\today}
\author{Angela Hicks, Samantha Miller-Brown}

\begin{abstract}
    The Schur functions, a basis for the symmetric polynomials (Sym), encode the irreducible representations of the symmetric group $\mathfrak{S}_n$ via the Frobenius characteristic map. In 1996, Krob and Thibon defined a quasisymmetric Frobenius map on the representations of $\mathcal{H}_n(0)$, mapping them to the quasisymmetric functions (QSym). Despite the obvious inclusion of Sym in QSym and the close relationship between $\mathfrak{S}_n$ and $\mathcal{H}_n(0)$, there is no known direct link between these two Frobenius characteristic maps and the related representations. We explore three specific situations in which a deformation of an $\mathfrak{S}_n$ action results in a valid $\mathcal{H}_n(0)$ action and gives a quasisymmetric Frobenius characteristic that is equal to the symmetric Frobenius characteristic. We introduce the concept of quasisymmetric compatibility, which formalizes a link between the two maps, and we show it applies to all $\mathfrak{S}_n$-modules.
\end{abstract}

\maketitle

\section{Introduction}

The Frobenius characteristic map $F_{char}$ a well-studied map on the representations of $\mathfrak{S}_n$, encodes modules as symmetric functions, allowing tools from linear algebra to be used to determine representation theoretic information about these $\mathfrak{S}_n$-modules. Here, we refer to the image of an $\mathfrak{S}_n$-module under this map as the symmetric Frobenius characteristic. 
In particular, if $S^\lambda$ is the irreducible $\mathfrak{S}_n$-module associated to the partition $\lambda$ of $n$, then $$F_{char}(S^{\lambda})=s_{\lambda},\footnote{By convention, the map is usually defined on class functions and sends the character of $S^{\lambda}$ to $s_\lambda$.  We chose to define the map directly on the modules here for consistency with the later maps of Krob and Thibon.}$$ where $s_{\lambda}$ is the Schur function associated to $\lambda$.  

In 1996, Krob and Thibon \cite{nsym} defined two similar maps, namely the quasisymmetric Frobenius characteristic $F_{char}^Q$ and the noncommutative Frobenius characteristic $F_{char}^N$, that encode modules of $\mathcal{H}_n(0)$, a deformation of the group algebra $\mathbb{C}[\mathfrak{S}_n]$, as quasisymmetric functions and noncommutative symmetric functions, respectively.  Since $\mathcal{H}_n(0)$ is not semisimple, indecomposable modules are not always irreducible. $F_{char}^N$, in particular, is defined only on the projective modules of $\mathcal{H}_n(0)$, sending the indecomposable modules $P_{\alpha}$, where $\alpha$ is a (strong) composition of $n$, to the noncommutative ribbon Schur functions:
$$F_{char}^N(P_{\alpha})=R_{\alpha}$$
$F_{char}^Q$ is defined more broadly on all (equivalence classes of) $\mathcal{H}_n(0)$-modules, and if $C_{\alpha}$ is the irreducible $\mathcal{H}_n(0)$-module associated to $\alpha$, then $$F_{char}^Q(C_{\alpha})=F_{\alpha},$$ where $F_{\alpha}$ is Gessel's fundamental quasisymmetric function associated to composition $\alpha$.  

When $M$ is a projective module, $$F_{char}^Q(M)=\chi(F_{char}^N(P_{\alpha})),$$
where $\chi$ is the forgetful map on noncommutative symmetric functions which allows the variables to commute: thus if $M$ is a projective $\mathcal{H}_n(0)$-module, it must have a symmetric image under $F_{char}^Q(M)$.  The converse is not the case, and in particular $F_{char}^Q$ is not injective.

The simple and projective indecomposible modules of $\mathcal{H}_n(0)$ were first classified by Norton~\cite{norton}. In 2016, Huang \cite{huang} introduced a tableau-based definition of the indecomposable module $P_{\alpha}$. See Figure \ref{indecomposable example} for one such example. A reader familiar with Specht modules may quickly notice the similarities to the Specht module associated to the ribbon Schur function $s_{\alpha}$, which we denote by $S^{\alpha}$. In particular, the underlying vector space of $S^{\alpha}$ is isomorphic to $P_{\alpha}$. While the symmetric action on the polytabloids of ribbon shape $\alpha$ is not identical to the $\mathcal{H}_n(0)$-action on the ribbon tableaux, there are clear similarities. 

\begin{figure}[h]
    \centering
\begin{tikzpicture}
    \ytableausetup{smalltableaux}
    \node (label) at (-4,2) {$\mathbf{P_{121}}$};
  \node (1423) at (0,2) {${\ytableaushort{2,14,\none3}}$};
  \node (1423 diamond) at (0,0) {${\ytableaushort{3,14,\none2}}$};
  \node (1324) at (2,-2) {${\ytableaushort{4,13,\none2}}$};
  \node (2413) at (-2,-2) {${\ytableaushort{3,24,\none1}}$};
  \node (bottom diamond) at (0,-4) {${\ytableaushort{4,23,\none1}}$};
  \node[scale=0.75] (bottom zero) at (1,-4) {0};
  \draw (1423) edge[->] node[left] {$\overline{\pi}_2$} (1423 diamond);
  \draw (1423 diamond) edge[->] node[right] {$\overline{\pi}_3$} (1324);
  \draw (1423 diamond) edge[->] node[left] {$\overline{\pi}_1$} (2413);
  \draw (2413) edge[->] node[left] {$\overline{\pi}_3$} (bottom diamond);
  \draw (1324) edge[->] node[right] {$\overline{\pi}_1$} (bottom diamond);
  \draw[scale=0.5] (bottom diamond) edge[->] node[scale=0.5,below] {$\overline{\pi}_2$} (bottom zero);
  \path (2413) edge[loop left] node[scale=0.75] (2413) {$\overline{\pi}_1=-1$} (2413);
  \path (1324) edge[loop right] node[scale=0.75] (1324) {$\overline{\pi}_3=-1$} (1324);
  \path (1423 diamond) edge[loop left] node[scale=0.75] (1423 diamond) {$\overline{\pi}_2=-1$} (1423 diamond);
  \path (bottom diamond) edge[loop left] node[scale=0.75] (bottom diamond) {$\overline{\pi}_1,\overline{\pi}_3=-1$} (bottom diamond);
  \path (1423) edge[loop right] node[scale=0.75] (1423) {$\overline{\pi}_1,\overline{\pi}_3=-1$} (1423);
\draw  (-5,3) rectangle (5,-5);
\end{tikzpicture} 
\caption{A Visual Representation of $P_{121}$}\label{indecomposable example}
\end{figure}

Since $F_{char}^N(P_{\alpha})=R_{\alpha}$ and $\chi(R_{\alpha})=s_{\alpha}$, $$F_{char}(S^{\alpha})=s_{\alpha}=\chi(F_{char}^N(P_{\alpha}))=F_{char}^Q(P_{\alpha}).$$
Similar coincidences have been observed elsewhere. For example, Huang and Rhoades in \cite{huang2018ordered} expanded on the quasisymmetric Frobenius image of a $\mathcal{H}_n(0)$-deformation of the generalized ring of coinvariants, showing that the result is the same symmetric polynomial as the symmetric Frobenius characteristic of the $\mathfrak{S}_n$-module.  While their construction of the $\mathcal{H}_n(0)$-module is specific to a quotient space of $\mathbb{C}[x_1, \dots, x_n]$ and not generalizable to all $\mathcal{H}_n(0)$-modules, these coincidences suggest one should think more generally about deformations of $\mathfrak{S}_n$-modules which result in the same quasisymmetric Frobenius image as the classical Frobenius image.

In particular, since the space of symmetric functions Sym is a subspace of the space of quasisymmetric functions QSym and $\mathcal{H}_n(0)$ is a deformation of $\mathfrak{S}_n$, a natural question is:

\begin{mquestion}
    What map $f$ causes the following diagram to commute, where $\iota$ denotes inclusion?

\begin{center}
\begin{tikzcd}
    \text{Rep'ns of $\mathfrak{S}_n$\footnotemark} \arrow[r, "F_{char}"] \arrow[d, dashed, "f"] & \text{Sym} \arrow[d, hookrightarrow, "\iota"] \\
    \text{Rep'ns of $\mathcal{H}_n(0)$} \arrow[r, "F_{char}^Q"]             & \text{QSym}
\end{tikzcd}
\footnotetext{Formally, $f$ should be a map from the Grothendieck group of $\mathbb{C}[\mathfrak{S}_n]$ to the Grothendieck group of $\mathcal{H}_n(0)$, as explained in Section \ref{sec: grothendieck}.}
\end{center}
\end{mquestion}

\noindent This paper provides an answer to this question.  Our work suggests and is inspired by a broader open question, which we do not attempt to answer, but which motivates our approach:
 \begin{motivation} Can one determine the symmetric Frobenius image of an $\mathfrak{S}_n$-module by computing the quasisymmetric Frobenius image of a related $\mathcal{H}_n(0)$-module?\end{motivation}  

In Section \ref{section: bg}, we give necessary background on the representations of $\mathfrak{S}_n$ and the representations of $\mathcal{H}_n(0)$. Next in Sections \ref{section: tabloids} and \ref{section: specht}, we give two $\mathfrak{S}_n$-modules and their corresponding $\mathcal{H}_n(0)$ modules. These two examples suggest a broader implication in Section \ref{section: broad impl}, where every $\mathfrak{S}_n$-module may be deformed in such a way. However, in Section \ref{section: coinv}, we explore our final example, where the implication and subsequent restrictions must be weakened. Finally, we conclude this paper in Section \ref{section: open} by comparing this paper to previous work, leading to two open questions that may be explored in future work.

\section{Background}\label{section: bg}

\subsection{The Symmetric Group $\mathfrak{S}_n$ and its Representations}

While we give a brief review in this section of relevant facts and notation regarding $\mathfrak{S}_n$ and its representations, we assume the reader has some familiarity with these 1423ics. The less familiar reader who is motivated to learn more is recommended to consult \cite{sagan} and \cite{stanley}.

Recall that the symmetric group $\mathfrak{S}_n$ is generated by the simple transpositions $s_i$, where 
\begin{align*}
    s_i^2 &=1 \text{ for } i\leq n-1 \\
    s_is_j&=s_js_i \text{ for } |i-j|\geq 2 \\
    s_is_{i+1}s_i&=s_{i+1}s_is_{i+1}.
\end{align*}
Given a permutation $w=w_1w_2\cdots w_n\in\mathfrak{S}_n$, we define the following statistics: let the \textit{descent set of $w$} be $$\text{des}(w)=\{i \mid w_i>w_{i+1}\},$$ the \textit{major index of $w$} to be $$\text{maj}(w)=\sum\limits_{i\in\text{des}(w)} i,$$ and the \textit{i-descent set of $w$} to be $$\text{ides}(w)=\{i\mid i \text{ occurs after } i+1 \text{ in } w\}.$$

Now, for $w=w_1w_2\cdots w_n$, let $w^{-1}$ represent the \textit{inverse} of the permutation $w$ and $\text{rev}(w)$ be the \textit{reverse} of $w$ such that $\text{rev}(w)=w_nw_{n-1}\cdots w_1$. With this, we offer an alternative definition of the i-descent set:
$$\text{ides}=\{i\mid w^{-1}_i>w^{-1}_{i+1}\}$$

Next, we introduce several basic notions in the representation theory of a finite group $G$. Recall that a well-defined action of $G$ on a vector space $M$ is a group homomorphism from $G$ to $GL(M)$. This group homomorphism is called a representation of $G$, and we say that $M$ is a $G$-module. Given a $G$-module $M$, we define the character of an element $g\in G$, written $\chi(g)$, by the trace of the linear transformation $M\rightarrow M$ defined by $v\mapsto gv$. The character $\chi$ is constant on conjugacy classes; in the case of $G=\mathfrak{S}_n$, the conjugacy classes are determined by the cycle types of permutations, which is always a partition of $n$. Thus, $\chi_{\mu}=\chi(\sigma)$, where $\sigma$ is any permutation of cycle type $\mu$. 

A submodule $A\subseteq M$ is a subspace of $M$ that is closed under the action of $G$. An \textit{irreducible submodule} is one where it has no non-trivial submodules. In \cite{sagan}, it is shown that any $\mathfrak{S}_n$-module can be written as a direct sum of its irreducible submodules, that is, the group algebra, $\mathbb{C}[\mathfrak{S}_n]$, is semisimple. The irreducible submodules of $\mathfrak{S}_n$ are usually referred to as the Specht modules $S^{\lambda}$.

\subsection{Combinatorial Objects}

In the representation theory of the symmetric group $\mathfrak{S}_n$, combinatorial objects like tableaux play an important role. 

Recall that a \textit{partition} $\mu=(\mu_1, \mu_2, \dots, \mu_k)$ of $n$, written $\mu\vdash n$, is a weakly decreasing sequence of positive integers that sum to $n$ and that $l(\mu)=k$ is the \textit{length of} $\mu$. Similarly, recall a \textit{weak composition} $\alpha=(\alpha_1, \alpha_2, \dots, \alpha_k)$ of $n$, written $\alpha\vDash_w n$, is a sequence of nonnegative integers that sum to $n$. In the case where all integers are strictly positive, a \textit{strong composition} is denoted by $\alpha\vDash n$. Partitions give rise to Young tableaux, whereby in French notation, we stack $\mu_1$ boxes in the bottom row, $\mu_2$ boxes in the second, etc. Given another partition $\lambda$, we say a Young tableau has content $\lambda$ if we fill the tableaux with $\lambda_1$ 1s, $\lambda_2$ 2s, etc. 

We say that a Young tableau of shape $\mu$ and content $\lambda$ is \textit{semistandard} if the filling is such that the rows are weakly increasing left to right and the columns are strictly increasing bottom to 1423. We denote the space of semistandard Young tableaux of shape $\mu$ and content $\lambda$ by SSYT($\mu, \lambda$), where the number of semistandard Young tableaux of shape $\mu$ and content $\lambda$ is counted by the Kostka numbers $K_{\mu\lambda}$. On the other hand, a Young tableau of shape $\mu$ and content $1^n$ is \textit{standard} if both the rows and the columns are strictly increasing. The space of standard Young tableaux of shape $\mu$ is denoted by SYT($\mu$), where $f^{\lambda}=|\text{SYT}(\mu)|$. 

\begin{example}
The following tableau is a semistandard Young tableau of shape $(3,1)$ and content $(2,1,1)$:
\ytableausetup{smalltableaux}
        $$p={{\ytableaushort{2\none, 113}}}$$ 
On the other hand, the following tableau is a standard Young tableau with the same shape:
$$S={{\ytableaushort{4\none, 123}}}$$
\end{example}

\noindent A useful fact, which we will use later, restricts the relative positions of $i$ and $i+1$:
\begin{fact}\label{fact: posns in std tab}
    Given a standard tableau $T$, $i$ must be directly below, to the left of, strictly southeast of, or strictly northwest of $i+1$.
\end{fact}

One important statistic on Young tableaux is the \textit{i-descent set}; given a tableau $t$, we say $$\text{ides}(t)=\{i\mid i\text{ is south of } i+1 \text{ in } t\}.$$ For the readers familiar with reading words, note that the definition of the i-descent set of a tableau is precisely the i-descent set of the reading word of $t$. Using this definition of the i-descent set, we define the \textit{major index of $t$}: $$\text{maj}(t)=\sum\limits_{i\in\text{ides}(t)} i.$$

Furthermore, we can define a partial ordering on Young tableaux with content $1^n$. In fact, it can be shown that this ordering is a total ordering on standard Young tableaux. 

\begin{definition}\label{def: col dom tableaux}
    Given two standard tableaux $S$ and $T$, we say that $S\prec T$ if and only if the smallest letter that lies in a different position is further north in $T$ than $S$.
\end{definition}
\noindent We observe that in standard tableaux $S\prec T$, the smallest letter must be strictly further north and weakly further west in $T$ than in $S$. For this reason, we use northwest to emphasize that we are comparing standard tableaux.

\begin{example}
Given the following two tableaux of shape $(3,2,1)$, $$\ytableausetup{smalltableaux} S={{\ytableaushort{4, 35,126}}} \text{ and } T={{\ytableaushort{6,24,135}}},$$ we have $S\prec T,$ since 2 is further to the northwest in $T$ than in $S$. 
\end{example}

We define a natural action of $\mathfrak{S}_n$ on tableaux by $s_it=s$, where $s$ is the tableau that results from swapping the letters $i$ and $i+1$ in the filling of $t$. Note that the $s_i$ action on a standard tableau $T$ does not necessarily guarantee that $S$ is a standard tableau.

\begin{lemma}\label{lemma: si col dom on std tab}
    If $T$ and $s_iT$ are standard tableaux, then $T\prec s_iT$ if and only if $i+1$ is strictly northwest of $i$ in $T$.  
\end{lemma}
\begin{proof}
    First, given standard tableau $T$, we note that $s_iT$ is standard only if $i$ is strictly northwest or southeast of $i+1$ by Fact \ref{fact: posns in std tab}.

    Now, if $i$ is strictly southeast of $i+1$ in $T$, then $i$ is strictly northwest of $i+1$ in $s_iT$. So, by definition, we have $T\prec s_iT$. 

    On the other hand, if $T\prec s_iT$ and $s_iT$ is standard, we can infer from the definition of dominance ordering that $i$ must be strictly northwest of $i+1$ in $s_iT$. Therefore, $i+1$ must be strictly northwest of $i$ in $T$, as desired.
\end{proof}

We can further define an equivalence relation on Young tableaux with content $1^n$ by saying that $t\sim_r s$ iff each row of $t$ contains the same letters as the corresponding row of $s$. The resulting equivalence classes are \textit{tabloids} of shape $\mu$, written $\{t\}$. Each equivalence class has a tableau that is strictly increasing in all rows, and we refer to that tableau as the tabloid representative. 

\begin{example}
    If $t=\ytableausetup{smalltableaux}{{\begin{ytableau}
        3 & \none & \none \\ 1 & 2 & 4
    \end{ytableau}}}$, then the tabloid $\{t\}=\ytableausetup{boxsize=normal,tabloids,smalltableaux} {{\begin{ytableau}
        3 & \none & \none \\ 1 & 2 & 4
    \end{ytableau}}}$ is the equivalence class:
    \ytableausetup{notabloids,smalltableaux}
    $$\{t\}=\left\{{{\begin{ytableau}
        3 & \none & \none \\ 1 & 2 & 4
    \end{ytableau}}}, {{\begin{ytableau}
        3 & \none & \none \\ 2 & 1 & 4
    \end{ytableau}}},{{\begin{ytableau}
        3 & \none & \none \\ 1 & 4 & 2
    \end{ytableau}}},{{\begin{ytableau}
        3 & \none & \none \\ 2 & 4 & 1
    \end{ytableau}}},{{\begin{ytableau}
        3 & \none & \none \\ 4 & 1 & 2
    \end{ytableau}}},{{\begin{ytableau}
        3 & \none & \none \\ 4 & 2 & 1
    \end{ytableau}}}\right\}$$
\end{example}

\noindent Going forward, we introduce two well-studied $\mathfrak{S}_n$-modules, which are each based on tabloids.

\begin{definition}
    Given a partition $\lambda$, let $M^{\lambda}$ be the vector space generated by tabloids of shape $\lambda$.
\end{definition}

\noindent We will make use of the following bijection on this basis of $M^{\lambda}$:

\begin{fact}\label{fact: tabloid bijection}
    Tabloids of shape $\lambda$ are in bijection with the set of words of type $\text{rev}(\lambda)$ $W(\text{rev}(\lambda))$. The straightforward bijection is given as $\{t\}\mapsto w$, where $w=w_1w_2\cdots w_n$ and $w_i=\text{l}(\lambda)-r_i+1$, where $r_i$ is the row of $\{t\}$ in which the letter $i$ sits.
\end{fact}

\begin{example}
    Consider the tabloid $\{t\}=\ytableausetup{boxsize=normal,tabloids,smalltableaux} \begin{ytableau}
        7 & \none & \none \\ 1 & 4 & 5 \\ 2 & 3 & 6 & 8
    \end{ytableau}$, then $w_t=23322313$.
\end{example}

Now, we give some statistics on tabloids, as well as some related results. Given a tabloid $\{t\}\in M^{\lambda}$, let $w_t$ be the word of content $\lambda$ defined by the bijection in Fact \ref{fact: tabloid bijection}. We give an analogous definition of i-descent sets on tabloids.

\begin{definition}\label{def: des of tabloid}
    For a tabloid $\{t\}\in M^{\lambda}$, define the i-descent set of $\{t\}$ as $$\text{ides}(\{t\})=\{i\mid i \text{ is lower than } i+1 \text{ in } t\}$$
\end{definition}

\noindent Our bijection on tabloids was chosen to have the following property:

\begin{lemma}\label{lem: bij tab des}
    By construction, for $\{t\}\in M^{\lambda}$ and corresponding $w_t\in W(\text{rev}(\lambda))$, we have $$\text{ides}(\{t\})=\text{des}(w_t).$$
\end{lemma}

We also define an analogous $s_i$ action on tabloids, where we swap the letters $i$ and $i+1$. Under this action, $M^{\lambda}$ is an $\mathfrak{S}_n$-module.

Finally, we make explicit a partial ordering on $M^{\lambda}$, by extending $\prec$ in Definition \ref{def: col dom tableaux}.

\begin{definition}
    Given tabloids $\{s\}$ and $\{t\}$ of shape $\lambda$, we say $\{s\}\prec\{t\}$  if the smallest letter in a different position is further north in $\{t\}$ than in $\{s\}$.
\end{definition}
\noindent For convenience, we often want a total ordering, rather than a partial ordering. Thus, extend this partial ordering to a total ordering $\prec_c$ on $M^{\lambda}$ arbitrarily.

Next, we consider a submodule of $M^{\lambda}$. In particular, the Specht module $S^{\lambda}$ is defined using a signed sum of tabloids. Given a tabloid representative $t$, we define a signed sum of tabloids based on the \textit{column space of} $t$, $\text{col}(t)$. The column space of $t$ is the cross product of sets of permutations that permute the letters within each column of $t$. Define a polytabloid, $e_t$, by a signed sum of tabloids:
$$e_t=\sum\limits_{\sigma\in\text{col}(t)} \text{sgn}(\sigma)\{\sigma t\}$$

$S^\lambda$, the vector space spanned by polytabloids of a given shape $\lambda$, has basis $\{e_T\}_{T\in\text{SYT}(\lambda)}$. 

We can define an $s_i$ action on polytabloids to ensure that we have an $\mathfrak{S}_n$-module. If we define this action on the standard polytabloids, we would say that $$s_ie_T=e_{s_iT}.$$ However, recall that $s_iT$ is not necessarily a standard tableau. Thus, we employ a straightening algorithm through the use of Garnir elements (see \cite{sagan} for a full description) to write the $s_i$ action in terms of the basis of $S^{\lambda}$. In particular, we see that $s_i(e_T)$ may be written as a linear combination of larger standard polytabloids with coefficients $c_{T,S}^i\in\mathbb{C}$:
 $$s_i(e_T)=\begin{cases}
-e_T & i, i+1 \hspace{2mm}\text{in the same column} \\
e_T\pm \sum\limits_{T\prec S} c_{T,S}^ie_{S} & i,i+1 \hspace{2mm}\text{in the same row}\\
e_{s_iT} & i, i+1 \text{ in different rows/columns}
\end{cases}$$

\begin{definition}[Specht module]
   For $\lambda\vdash n$, the Specht module $S^{\lambda}$ is the submodule of $M^{\lambda}$ with basis $\{e_T\mid T\in\text{SYT}(\lambda)\}$.
\end{definition}

These Specht modules form the irreducible representations of $\mathfrak{S}_n$. Thus, any $\mathfrak{S}_n$-module $M$ is isomorphic to a direct sum of Specht modules: $$M\cong \bigoplus\limits_{\lambda\vdash n} c_{\lambda}S^{\lambda}$$

On this basis, we can once again define a total ordering that is analogous to the total ordering on standard Young tableaux.

\begin{definition}\label{def: col dom poly}
    Given two polytabloids, we say $e_S\prec e_T$ if and only if $S\prec T$. 
\end{definition}

\subsection{Symmetric Functions}

In this paper, we consider two subspaces of $\mathbb{C}[x_1, \dots, x_n]$. $\mathfrak{S}_n$ acts on $\mathbb{C}[x_1, \dots, x_n]$ by swapping indices such that, for $1<i\leq n$, $$s_if(x_1, \dots, x_i, x_{i+1}, \dots, x_n)=f(x_1, \dots, x_{i+1}, x_i, \dots, x_n).$$

The first subspace is the space of \textit{symmetric functions} $\text{Sym}_n$. A polynomial $f\in\mathbb{C}[x_1, \dots, x_n]$ is considered \textit{symmetric} if, for all $1\leq i< n$, $$s_if(x_1, \dots, x_n)=f(x_1,\dots, x_n).$$ The symmetric functions form a graded subspace of $\mathbb{C}[x_1,\dots, x_n]$, and we list some of the relevant bases here. 

\begin{definition}[Bases of symmetric functions]
        The \textit{power sum}, the \textit{complete homogeneous}, and the \textit{elementary} symmetric functions are defined multiplicatively, so that $p_{\mu}=p_{\mu_1}p_{\mu_2}\cdots p_{\mu_l}$, $h_{\mu}=h_{\mu_1}h_{\mu_2}\cdots h_{\mu_l}$, and $e_{\mu}=e_{\mu_1}e_{\mu_2}\cdots e_{\mu_l}$, where $$ p_{k}(x_1, \dots, x_n)=\sum\limits_{i=1}^n x_i^k,$$ $$ h_{k}(x_1, \dots, x_n)=\sum\limits_{1\leq i_1\leq\cdots\leq i_k\leq n} x_{i_1}\cdots x_{i_k}, $$ and $$e_k(x_1, \dots, x_n)=\sum\limits_{1< i_1<\cdots< i_k\leq n} x_{i_1}\cdots x_{i_k}.$$
\end{definition}

\begin{definition}[Schur functions]
    The \textit{Schur functions} are defined using semi-standard Young tableaux: $$s_{\lambda}=\sum\limits_{T\in\text{SSYT}(\lambda)} x^T$$
\end{definition}

We will need the relationships between some of these bases in this paper. For relevant proofs, see \cite{stanley}.

\begin{lemma}\label{lem: h in terms of schur}
    $$h_{\lambda}=\sum\limits_{\mu\vdash n} K_{\mu\lambda}s_{\mu}.$$
\end{lemma}

Given any $n$, it turns out that $\text{Sym}_n$ is isomorphic to the vector space of $n$-dimensional representations of $\mathfrak{S}_n$. To see this, we introduce the \textit{Frobenius characteristic map}, which employs the use of the character of a permutation. Going forward, we refer to this as the symmetric Frobenius characteristic image.

\begin{definition}[Frobenius characteristic map]
    The Frobenius characteristic of an $\mathfrak{S}_n$-module $M$ is defined using the character function $\chi:\mathfrak{S}_n\rightarrow\mathbb{C}$, where $\chi_{\mu}$ is the image of a permutation with cycle type $\mu$ under $\chi$:
        $$F_{char}(M)=\sum\limits_{\mu\vdash n} \chi_{\mu}\frac{p_{\mu}}{z_{\mu}}$$
\end{definition}
We note that, if $M$ is graded module, then one may encode the degree of the submodules using a graded version of the Frobenius characteristic, $F_{char}(M;q)$.

Using the above definition, we obtain the following results about the modules defined in the previous section (see \cite{sagan} for relevant proofs):

$$F_{char}(S^{\lambda})=s_{\lambda}$$
$$F_{char}(M^{\lambda})=h_{\lambda}$$

\subsection{Quasisymmetric Functions}

The second subspace of formal power series that we will use often is the the quasisymmetric functions $\text{QSym}_n$. Like Sym, QSym is a sub-Hopf algebra of $\mathbb{C}[x_1, \dots, x_n]$. Quasisymmetric functions are \textit{shift invariant}, so the coefficient of $x_1^{\alpha_1}\cdots x_k^{\alpha_k}$ is equal to the coefficient of $x_{i_1}^{\alpha_1}\cdots x_{i_k}^{\alpha_k}$, where $i_1<i_2<\cdots< i_k$.

\begin{definition}[Gessel's fundamental quasisymmetric functions]
    Define the Gessel's fundamental quasisymmetric function $F_K$ as the following sum:
    $$F_K=\sum\limits_{\substack{1\leq i_1\leq i_2\leq\cdots\leq i_k \\ j\in K\Rightarrow i_j<i_{j+1}}} x_{i_1}x_{i_2}\cdots x_{i_k}$$.
\end{definition}

\begin{example}
$$F_{\{1\}}(x_1, x_2,x_3)=x_{1} x_{2}^{2} + x_{1} x_{2} x_{3} + x_{1} x_{3}^{2} + x_{2} x_{3}^{2}$$
\end{example}

Every symmetric function is shift invariant, so $\text{Sym}_n\subset \text{QSym}_n$. As a result, we are able to write any symmetric functions in terms of the fundamental quasisymmetric functions. 

\begin{fact}\label{fact: schurs and fundamentals}
    We can write the Schur functions as the following sum:
$$s_{\lambda}=\sum\limits_{T\in \text{SYT}(\lambda)} F_{\text{ides}(T)}$$
\end{fact}

\subsection{The 0-Hecke Algebra $\mathcal{H}_n(0)$ and its Representations}

The 0-Hecke algebra is a deformation of $\mathfrak{S}_n$ with generators $\pibar_i$ that satisfy the following relations.
\begin{align*}
    \overline{\pi}_i^2 &=-\overline{\pi}_i \text{ for } i\leq n-1 \\
    \overline{\pi}_i\overline{\pi}_j&=\overline{\pi}_j\overline{\pi}_i \text{ for } |i-j|\geq 2 \\
    \overline{\pi}_i\overline{\pi}_{i+1}\overline{\pi}_i&=\overline{\pi}_{i+1}\overline{\pi}_i\overline{\pi}_{i+1}
\end{align*}
Given $w\in\mathfrak{S}_n$ written as a product of simple transpositions, we can define $\pibar_w$ multiplicatively.

The irreducible representations of $\mathcal{H}_n(0)$ are one-dimensional submodules naturally indexed by subsets of $[n-1]$. We can define these one-dimensional irreducibles using the actions of the generators $\pibar_i$.

\begin{definition}[Irreducible $\mathcal{H}_n(0)$-modules]
    For $I\subseteq[n-1]$, let $C_{I}$ be the one-dimensional $\mathcal{H}_n(0)$ representation defined by $$C_{I}(\overline{\pi}_i)=\begin{cases}
        0 & i\notin I \\
        -1 & i\in I
    \end{cases}$$
\end{definition}
\noindent In \cite{norton}, it is shown that $\{C_{I} \mid I\subseteq[n-1]\}$ are all of the $2^{n-1}$ irreducible representations of $\mathcal{H}_n(0)$. 

Using this characterization of the irreducible submodules, Krob and Thibon \cite{nsym} define the following characteristic map:

\begin{definition}[Quasisymmetric Frobenius characteristic]
    Given an irreducible $\mathcal{H}_n(0)$-submodule $C_I$, define the quasisymmetric Frobenius characteristic map using fundamental quasisymmetric functions: $$F_{char}^Q(C_I)=F_I$$

    Now, given an arbitrary $\mathcal{H}_n(0)$-module $M$, let $M=M_1\supset M_2\supset\cdots M_k=\emptyset$ be a finite composition series for $M$. Then, since each $\faktor{M_j}{M_{j+1}}$ is irreducible, define $$F_{char}^Q(M)=\sum\limits_{j=1}^{k-1} F_{char}^Q\left(\faktor{M_j}{M_{j+1}}\right).$$
\end{definition}
\noindent By the Jordan-H{\"o}lder Theorem, the collection of isomorphism classes of these irreducible quotients is uniquely determined. Moreover, since each quotient is irreducible, it is isomorphic to $C_I$ for some $I$. Moving forward, we use $[\cdot]_j$ to denote equivalence classes in $\faktor{M_j}{M_{j+1}}$. Then, if $[v]_j$ is the generator of $\faktor{M_j}{M_{j+1}}$, $$F_{char}^Q\left(\faktor{M_j}{M_{j+1}}\right)=F_{\{i\mid \pibar_iv=-v\}}.$$

\subsection{Grothendieck Groups}\label{sec: grothendieck}

Through this point, we have two Frobenius characteristic maps: the symmetric Frobenius map on the representations of $\mathfrak{S}_n$, $F_{char}$, and the quasisymmetric Frobenius map on the representations of $\mathcal{H}_n(0)$, $F_{char}^Q$. In this section, we offer a more formal, and uniform, definition of these two maps.

\begin{definition}[Grothendieck group]
    For an algebra $A$, consider the $\mathbb{Z}$-vector space of isomorphism classes of finite-dimensional $A$-modules $\mathcal{R}(A)$. Let $[M]$ denote the isomorphism class of $M$. Then, the Grothendieck group $\mathcal{G}_0(A)$ is the quotient space of $\mathcal{R}(A)$ obtained by the relation $[M]=[M_1]+[M_2]$ when we have the short exact sequence $0\rightarrow M_1\rightarrow M\rightarrow M_2\rightarrow 0$.
\end{definition}

This simplifies considerably if the algebra is semisimple, as in the case of $\mathbb{C}[\mathfrak{S}_n]$, the group algebra of $\mathfrak{S}_n$. For any algebra, when $0\rightarrow M_1\xrightarrow{f} M\xrightarrow{g} M_2\rightarrow 0$ is a short exact sequence, we have $M_2\cong \faktor{M}{\text{im}(f)}$ with $\text{im}(f)\cong M_1$ a submodule of $M$. If the algebra is semisimple, however, then this short exact sequence always splits so that $M\cong M_1\oplus M_2$. Thus, if the algebra $A$ is semisimple, the relation defining $\mathcal{G}_0(A)$ is simply $[M]=[M_1]+[M_2]$, when $M\cong M_1\oplus M_2$, without the added complexity of short exact sequences.

Since $\mathcal{H}_n(0)$ is not semisimple, a finitely-generated module $M$ cannot necessarily be written as a direct sum of its irreducible submodules. The irreducible representations $C_I$ form the basis of the Grothendieck group of $\mathcal{H}_n(0)$ $\mathcal{G}_0(\mathcal{H}_n(0))$. Using this and the Specht modules $S^{\lambda}$ as the basis for $\mathcal{G}_0(\mathbb{C}[\mathfrak{S}_n])$, one can formalize the definitions of the symmetric Frobenius characteristic map and the quasisymmetric Frobenius characteristic map. In particular, these linear transformations are defined by
\begin{align*}
    F_{char}: \mathcal{G}_0\left(\mathbb{C}[\mathfrak{S}_n]\right) &\rightarrow \text{Sym}^n \\
    [S^{\lambda}] &\mapsto s_{\lambda}
\end{align*} 
and \begin{align*}
    F_{char}^Q:\mathcal{G}_0\left(\mathcal{H}_n(0)\right) &\rightarrow \text{QSym}^n \\
    [C_I] &\mapsto F_I
\end{align*}

Readers wishing for a first concrete example of applying this map are encouraged to consult \cite{tewari}, where Tewari and van Willigenburg give the nice example of a $\mathcal{H}_n(0)$-module whose quasisymmetric Frobenius image is the quasischur basis for the quasisymmetric functions.

\section{An $\mathcal{H}_n(0)$-Module, $\widehat{M^{\lambda}}$}\label{section: tabloids}

Our goal for this section is to deform the $\mathfrak{S}_n$-action on the permutation module $M^{\lambda}$ to create a new $\mathcal{H}_n(0)$-module $\widehat{M^{\lambda}}$ such that $$F_{char}^Q(\widehat{M^{\lambda}})=F_{char}(M^{\lambda}).$$ 

To define a $\mathcal{H}_n(0)$-action on $\widehat{M^{\lambda}}$, we use the vector space $M^{22}$ as a guiding example. Recall the $\mathfrak{S}_4$ action on $M^{22}$ on the left of Figure \ref{fig: M22 H0 action}, drawn such that if $\{s\}$ lies below $\{t\}$, then $\{s\}\prec\{t\}$. Now for a valid $\mathcal{H}_n(0)$-action, we must have $\overline{\pi}_i^2=-\overline{\pi}_i$, and thus all edges must be directed. To this end, we choose one direction in the diagram and add loops when necessary. To ensure that the quasisymmetric Frobenius characteristic is the same as the symmetric Frobenius characteristic, we will direct the $\pibar_i$ edges so that a non-trivial $\pibar_i$ action on a tabloid creates a new i-descent. In particular, consider the right of Figure \ref{fig: M22 H0 action}, where we direct the edges as desired, while adding ``negative" loops to satisfy the first commuting relation. Furthermore, note that all previous $s_i$ actions resulting in $s_iv=v$ cannot be analogously replaced by $\pibar_iv=v$ and must be adjusted. 

\begin{figure}[h]
\centering
\ytableausetup{boxsize=normal,tabloids}
\begin{tikzpicture}[scale=.75]
  \node (1423) at (2,0) {\tiny${\ytableaushort{14,23}}$};
  \node (2413) at (0,-2) {\tiny${\ytableaushort{24,13}}$};
  \node (1324) at (0,2) {\tiny${\ytableaushort{13,24}}$};
  \node (1234) at (0,4) {\tiny${\ytableaushort{12,34}}$};
  \node (3412) at (0,-4) {\tiny${\ytableaushort{34,12}}$};
  \node (2314) at (-2,0) {\tiny${\ytableaushort{23,14}}$};
  \draw (1423) edge[<->] node[right] {$s_1$} (2413);
  \draw (1423) edge[<->] node[right] {$s_3$} (1324);
  \draw (2413) edge[<->] node[left] {$s_2$} (3412);
  \draw (2413) edge[<->] node[right] {$s_3$} (2314);
  \draw (1324) edge[<->] node[right] {$s_2$} (1234);
  \draw (1324) edge[<->] node[left] {$s_1$} (2314);
  \path (2314) edge[loop left] node[scale=0.75] (2314) {$s_2=1$} (2314);
  \path (3412) edge[loop left] node[scale=0.75] (3412) {$s_1=s_3=1$} (3412);
  \path (1234) edge[loop left] node[scale=0.75] (1234) {$s_1=s_3=1$} (1234);
  \path (1423) edge[loop right] node[scale=0.75] (1423) {$s_2=1$} (1423);
\end{tikzpicture}
\begin{tikzpicture}[scale=.75]
    \ytableausetup{boxsize=normal,tabloids}
  \node (1423) at (2,0) {\tiny${\ytableaushort{14,23}}$};
  \node (2413) at (0,-2) {\tiny${\ytableaushort{24,13}}$};
  \node (1324) at (0,2) {\tiny${\ytableaushort{13,24}}$};
  
  \node (1234) at (0,4) {\tiny${\ytableaushort{12,34}}$};
  \node (3412) at (0,-4) {\tiny${\ytableaushort{34,12}}$};
  \node (2314) at (-2,0) {\tiny${\ytableaushort{23,14}}$};
  \node[scale=0.5] (mid zero) at (0,0) {0};
  \node[scale=0.5] (top zero) at (-1.5,4) {0};
  \node[scale=0.5] (3412 zero) at (1.5,-4) {0};
  \draw (1423) edge[->] node[left] {$\overline{\pi}_1$} (2413);
  \draw (1324) edge[->] node[right] {$\overline{\pi}_3$} (1423);
  \draw (2413) edge[->] node[left] {$\overline{\pi}_2$} (3412);
  \draw (2314) edge[->] node[left] {$\overline{\pi}_3$} (2413);
  \draw (1234) edge[->] node[right] {$\overline{\pi}_2$} (1324);
  \draw (1324) edge[->] node[right] {$\overline{\pi}_1$} (2314);
  \draw (1423) edge[->] node[scale=0.5,below] {$\overline{\pi}_2$} (mid zero);
  \draw[scale=0.5] (2314) edge[->] node[scale=0.5,below] {$\overline{\pi}_2$} (mid zero);
  \draw[scale=0.5] (3412) edge[->] node[scale=0.5,below] {$\overline{\pi}_1,\overline{\pi}_3$} (3412 zero);
  \draw[scale=0.5] (1234) edge[->] node[scale=0.5,below] {$\overline{\pi}_1,\overline{\pi}_3$} (top zero);
  \path (2413) edge[loop right] node[scale=0.75] (2413) {$\overline{\pi}_1=\overline{\pi}_3=-1$} (2413);
  \path (1324) edge[loop right] node[scale=0.75] (1324) {$\overline{\pi}_2=-1$} (1324);
  \path (3412) edge[loop left] node[scale=0.75] (3412) {$\overline{\pi}_2=-1$} (3412);
  \path (2314) edge[loop left] node[scale=0.75] (2314) {$\overline{\pi}_1=-1$} (2314);
  \path (1423) edge[loop right] node[scale=0.75] (2314) {$\overline{\pi}_3=-1$} (1423);
\end{tikzpicture}
\caption{An $\mathfrak{S}_4$-action on $M^{22}$ and an $\mathcal{H}_4(0)$-action on $\widehat{M^{22}}$}\label{fig: M22 H0 action}
\end{figure}

With this figure in mind, we define the following $\mathcal{H}_n(0)$-action on the tabloids of $\widehat{M^{\lambda}}$:

\begin{definition}\label{def: easy m action}
$$\overline{\pi}_{i}\{t\} = \begin{cases}
        0 & i, i+1 \text{ in the same row}\\
        -\{t\} & i \text{ south of } i+1 \\
        s_i\{t\} & i \text{ north of } i+1
        \end{cases}$$
\end{definition}

\noindent The next three propositions will show that this action is a valid $\mathcal{H}_n(0)$-action.

\begin{proposition}\label{prop: m squared}
    The action defined on $\widehat{M^{\lambda}}$ satisfies $\overline{\pi}_i^2=-\overline{\pi}_i$.
\end{proposition}
\begin{proof}
    We have three cases to consider, based on the relative positions of $i$ and $i+1$ in an arbitrary tabloid $\{t\}$. 

    For the first and shortest case, suppose $i$ and $i+1$ are in the same row in $\{t\}$. Then, $s_i\{t\}=\{t\}$. Thus, $\overline{\pi}_i\{t\}=0$. This is easy to see that $\overline{\pi}_i^2=-\overline{\pi}_i$ trivially.

    For the second case, suppose $i$ is north of $i+1$ in $\{t\}$. Then, $\overline{\pi}_i\{t\}=s_i\{t\}$, by definition of the action. In $s_i\{t\}$, $i$ is now south of $i+1$, so $\overline{\pi}_i^2\{t\}=\overline{\pi}_i(s_i\{t\})=-s_i\{t\}=-\overline{\pi}_i\{t\}$. 

    For the third and final case, suppose $i$ is south of $i+1$. Then, in $s_i\{t\}$, $i$ is north of $i+1$. Thus, by definition, $\overline{\pi}_i\{t\}=-\{t\}$, and $\overline{\pi}_i(-\{t\})=\{t\}=-\overline{\pi}_i\{t\}$. 
\end{proof}

\begin{proposition}\label{prop: m length 2}
    The action defined on $\widehat{M^{\lambda}}$ satisfies $\overline{\pi}_i\overline{\pi}_j=\overline{\pi}_j\overline{\pi}_i$ when $|i-j|\geq 2$. 
\end{proposition}
\begin{proof}
    It is straightforward to see that the relative positions of $i$ and $i+1$ and the relative positions of $j$ and $j+1$ have no effect on $\overline{\pi}_j$ and $\overline{\pi}_i$, respectively. 
\end{proof}

\begin{proposition}\label{prop: m length 3}
    The action defined on $\widehat{M^{\lambda}}$ satisfies $\overline{\pi}_i\overline{\pi}_{i+1}\overline{\pi}_i=\overline{\pi}_{i+1}\overline{\pi}_i\overline{\pi}_{i+1}$. 
\end{proposition}
\begin{proof}
    While there are 12 individual cases to check for an arbitrary tabloid $\{t\}$, across three families, we proceed by showing one illustrative case from each of the three families. The remaining cases are similar to these illustrative cases and are straightforward computations.

\begin{case}
    For the first representative case, we assume that, in a tabloid $\{t\}$, $i+2$ is above $i$, which is above $i+1$. Then, we have
    \begin{align*}
        \pibar_i\pibar_{i+1}\pibar_i\{t\} &= \pibar_i\pibar_{i+1}(s_i\{t\}) \\
        &= \pibar_i(-s_i\{t\}) \\
        &= s_i\{t\} \\
        \pibar_{i+1}\pibar_i\pibar_{i+1}\{t\} &= \pibar_{i+1}\pibar_i(-\{t\}) \\
        &= \pibar_{i+1}(-s_i\{t\}) \\
        &= s_i\{t\}
    \end{align*}
\end{case}

\begin{case}
    Next, suppose $i$ and $i+2$ are in the same row, while $i+1$ is in a row above. Then, 
    \begin{align*}
        \pibar_i\pibar_{i+1}\pibar_i\{t\} &= \pibar_i\pibar_{i+1}(-\{t\}) \\
        &= \pibar_i(-s_{i+1}\{t\}) \\
        &= 0 \\
        \pibar_{i+1}\pibar_i\pibar_{i+1}\{t\} &= \pibar_{i+1}\pibar_i(s_{i+1}\{t\}) \\
        &= 0
    \end{align*}
\end{case}

\begin{case}
    Finally, suppose $i$ and $i+2$ are in the same row, while $i+1$ is in a row below. Then, 
    \begin{align*}
        \pibar_i\pibar_{i+1}\pibar_i\{t\} &= \pibar_i\pibar_{i+1}(s_i\{t\}) \\
        &= 0 \\
        \pibar_{i+1}\pibar_i\pibar_{i+1}\{t\} &= \pibar_{i+1}\pibar_i(-\{t\}) \\
        &= \pibar_{i+1}(-s_i\{t\}) \\
        &= 0
    \end{align*}
\end{case}
\end{proof}

Now that we have shown that this action is valid, we can start creating the composition series that will lead us to the quasisymmetric Frobenius characteristic. 

\begin{figure}[h]
    \centering
\begin{tikzpicture}
    \node (label) at (-4.4,4.5) {$\widehat{\mathbf{M^{22}}}=\mathbf{M_1}$};
    \node at (-4.7,2.2) {$\mathbf{M_2}$};
    \node at (-4.5,0.4) {$\mathbf{M_3}$};
  \node (1423) at (2,0) {\tiny${\ytableaushort{14,23}}$};
  \node (2413) at (0,-2) {\tiny${\ytableaushort{24,13}}$};
  \node (1324) at (0,2) {\tiny${\ytableaushort{13,24}}$};
  
  \node (1234) at (0,4) {\tiny${\ytableaushort{12,34}}$};
  \node (3412) at (0,-4) {\tiny${\ytableaushort{34,12}}$};
  \node (2314) at (-2,0) {\tiny${\ytableaushort{23,14}}$};
  \node[scale=0.5] (mid zero) at (0,0) {0};
  \node[scale=0.5] (top zero) at (-1,4) {0};
  \node[scale=0.5] (3412 zero) at (1,-4) {0};
  \draw (1423) edge[->] node[left] {$\overline{\pi}_1$} (2413);
  \draw (1324) edge[->] node[right] {$\overline{\pi}_3$} (1423);
  \draw (2413) edge[->] node[left] {$\overline{\pi}_2$} (3412);
  \draw (2314) edge[->] node[left] {$\overline{\pi}_3$} (2413);
  \draw (1234) edge[->] node[right] {$\overline{\pi}_2$} (1324);
  \draw (1324) edge[->] node[right] {$\overline{\pi}_1$} (2314);
  \draw (1423) edge[->] node[scale=0.5,below] {$\overline{\pi}_2$} (mid zero);
  \draw[scale=0.5] (2314) edge[->] node[scale=0.5,below] {$\overline{\pi}_2$} (mid zero);
  \draw[scale=0.5] (3412) edge[->] node[scale=0.5,below] {$\overline{\pi}_1,\overline{\pi}_3$} (3412 zero);
  \draw[scale=0.5] (1234) edge[->] node[scale=0.5,below] {$\overline{\pi}_1,\overline{\pi}_3$} (top zero);
  \path (2413) edge[loop right] node[scale=0.75] (2413) {$\overline{\pi}_1=\overline{\pi}_3=-1$} (2413);
  \path (1324) edge[loop right] node[scale=0.75] (1324) {$\overline{\pi}_2=-1$} (1324);
  \path (3412) edge[loop left] node[scale=0.75] (3412) {$\overline{\pi}_2=-1$} (3412);
  \path (2314) edge[loop left] node[scale=0.75] (2314) {$\overline{\pi}_1=-1$} (2314);
  \path (1423) edge[loop right] node[scale=0.75] (2314) {$\overline{\pi}_3=-1$} (1423);
\draw  (-5.5,5) rectangle (6,-5);
\draw [dashed] (-5.25,2.7) rectangle (5.6,-4.75);
\draw [dotted] (-5,-4.6) rectangle (5.2,0.75);
\end{tikzpicture} 
\caption{Part of the Composition Series of $\widehat{M^{22}}$}\label{comp series of M22}
\end{figure}

In Figure \ref{comp series of M22}, we draw the first three submodules in the composition series $\widehat{M^{\lambda}}=M_1\supset M_2\supset\cdots M_k=\emptyset$. We see that $\faktor{M_i}{M_{i+1}}$ is indeed one-dimensional, generated by a single tabloid from the basis. If we consider the figure further, we can observe that we remove tabloids from largest to smallest according to $\prec_c$. In particular, we can also define the following action that we claim is equivalent to Definition \ref{def: easy m action}.

\begin{definition}\label{def: ordered m action}
        $$\overline{\pi}_{i}\{t\} = \begin{cases}
        0 & s_i\{t\}=\{t\}\\
-\{t\} & \{t\}\prec s_i\{t\} \\
s_i\{t\} & s_i\{t\}\prec \{t\}
\end{cases}$$
\end{definition}

It is straightforward to see that Definition \ref{def: easy m action} and Definition \ref{def: ordered m action} are equivalent. Thus, we can assume that the commuting relations hold for this action, as well. In particular, this action will make it easier to define the composition series, so that each $M_i$ is determined by removing tabloids that are greater according to $\prec_c$. Again, referring to Figure \ref{comp series of M22}, we can work through the quasisymmetric Frobenius characteristic map.

\begin{example}[Frobenius Characteristic of $\widehat{M^{22}}$]
    Given $\faktor{M_j}{M_{j+1}}\cong C_I$ with generator $[v]_j$, we have $I=\{i\mid \overline{\pi}_iv=-v\}$. Thus, $$F_{char}^Q(\faktor{M_j}{M_{j+1}})=F_I.$$

    Again, from Figure \ref{comp series of M22}, we can see that each successive quotient module $\faktor{M_j}{M_{j+1}}$ is generated by a single equivalence class $[\{t\}_j]_j$. So, for each tabloid, we have $I=\{i\mid \pibar_i\{t\}_j=-\{t\}_j\}$. Thus, for each tabloid, the ``negative" loops that appear will index the fundamental quasisymmetric function that appears in the quasisymmetric Frobenius characteristic image of $\widehat{M^{22}}$. However, we know that a ``negative" loop occurs when we would remove an i-descent, so $I=\text{ides}(\{t\})$.
    
    So, given the figure above, we have $$F_{char}^Q(\widehat{M^{22}})=F_{\emptyset}+F_{\{1\}}+F_{\{3\}}+2F_{\{2\}}+F_{\{1,3\}}$$
    
    However, this sum of quasisymmetric functions is symmetric and can be written in terms of a basis of Sym. In particular, we have $$F_{char}^Q(\widehat{M^{22}})=h_{22}$$
\end{example}

For the upcoming result, we will need a result of the RSK correspondence, which we state here for use later.

\begin{fact}\label{fact: RSK ssyt ides}
    Given a tabloid $\{t\}\in M^{\lambda}$, let $w_t$ be the word that corresponds to the tabloid. 

    Under RSK, $w_t$ maps to the pair $(P,Q)$ with $P$, a semistandard Young tableau, and $Q$, a standard Young tableau, where $$\text{ides}(Q)=\text{des}(w_t)=\text{ides}(\{t\})$$
\end{fact}

As anticipated, we have the following result.

\begin{theorem}\label{thm: quasi char of m}
There exists a $\mathcal{H}_n(0)$-action on the permutation modules $\widehat{M^{\lambda}}$ such that $$F_{char}^Q(\widehat{M^{\lambda}})=h_{\lambda}=F_{char}(M^{\lambda}).$$
\end{theorem}

\begin{proof}
    Let $\{t\}_1, \dots, \{t\}_k$ be the list of tabloids that form the basis of $\widehat{M^{\lambda}}$ given in the total ordering consistent with dominance ordering, $\prec_c$, where $\{t\}_1$ is the largest. Define submodules $M_j=\text{span}(\{t\}_j, \{t\}_{j+1}, \dots, \{t\}_k)$. Thus, we have a composition series $\widehat{M^{\lambda}}=M_1\supset M_2\supset\cdots\supset M_k=\emptyset$.
    
    We note a couple of facts about this composition series. First, we observe that $\faktor{M_j}{M_{j+1}}$ is generated by $[\{t\}_j]_j$. Furthermore, this composition series is, in fact, a composition series of submodules as we have $M_j\supset M_{j+1}$, by construction of our action.
    
    Then, $\faktor{M_j}{M_{j+1}}\cong C_I$, where $C_I=\{i\mid \pibar_i\{t\}_j=-\{t\}_j\}$. So, by the definition of the quasisymmetric Frobenius characteristic map, we have $$F_{char}^Q(\widehat{M^{\lambda}})=\sum\limits_{\{t\}\in \widehat{M^{\lambda}}} F_{\{i\mid \pibar_i\{t\}=-\{t\}\}}$$
    
    However, since $\overline{\pi}_i\{t\}=-\{t\}$ when $i$ is south of $i+1$, this means that $I=\{i\mid i \text{ south of } i+1\text{ in } t\}$. By definition of the i-descent set of a tabloid, we can conclude that $I=\text{ides}(\{t\})$. Thus, overall, we have the following formula for the quasisymmetric Frobenius characteristic of $\widehat{M^{\lambda}}$:
    \begin{align*}
        F_{char}^Q(\widehat{M^{\lambda}}) &= \sum\limits_{\{t\}\in \widehat{M^{\lambda}}} F_{\text{ides}(\{t\})} \\
        &=\sum\limits_{w\in W(\text{rev}(\lambda))} F_{\text{des}(w)}
    \end{align*}
    where the second sum comes from the bijection on tabloids given in Fact \ref{fact: tabloid bijection} and Lemma \ref{lem: bij tab des}. 
    
    Splitting the sum, we have $$F_{char}^Q(\widehat{M^{\lambda}}) =\sum\limits_{B\subseteq [n-1]} F_B \sum\limits_{w\in W(\text{rev}(\lambda))} \mathbbm{1}_{\text{des}(w)=B},$$
    and by a result of the RSK correspondence given in Fact \ref{fact: RSK ssyt ides}, we have $$F_{char}^Q(\widehat{M^{\lambda}})=\sum\limits_{B\subseteq [n-1]} F_B \sum\limits_{\mu\vdash n}\sum\limits_{P\in\text{SSYT}(\mu, \text{rev}(\lambda))}\sum\limits_{Q\in\text{SYT}(\mu)} \mathbbm{1}_{\text{ides}(Q)=B}.$$
    
    Straightforward manipulations will then show that 

        \begin{align*}
F_{char}^Q(\widehat{M^{\lambda}}) &= \sum\limits_{B\subseteq[n-1]} \sum\limits_{\mu\vdash n}\sum\limits_{Q\in\text{SYT}(\mu)} K_{\mu\text{rev}(\lambda)} \mathbbm{1}_{\text{ides}(Q)=B}F_B \\
&= \sum\limits_{\mu\vdash n}\sum\limits_{Q\in\text{SYT}(\mu)} K_{\mu\text{rev}(\lambda)} F_{\text{ides}(Q)} \\
&= \sum\limits_{\mu\vdash n} K_{\mu\text{rev}(\lambda)}\sum\limits_{Q\in\text{SYT}(\mu)} F_{\text{ides}(Q)} \\
&= \sum\limits_{\mu\vdash n} K_{\mu\text{rev}(\lambda)} s_{\mu}
        \end{align*}

        In his textbook \cite{sagan}, Sagan takes a representation theoretic approach to showing that $K_{\mu\tilde{\lambda}}=K_{\mu\lambda}$ for any rearrangement $\tilde{\lambda}$ of $\lambda$. Thus, 
        \begin{align*}
            F_{char}^Q(\widehat{M^{\lambda}}) &= \sum\limits_{\mu\vdash n} K_{\mu\text{rev}(\lambda)} s_{\mu} \\
            &= \sum\limits_{\mu\vdash n} K_{\mu\lambda} s_{\mu} \\
            &= h_{\lambda} 
        \end{align*}

\end{proof}

\section{An $\mathcal{H}_n(0)$-module, $\widehat{S^{\lambda}}$}\label{section: specht}

Given a valid $\mathcal{H}_n(0)$-action on tabloids, we can extend to a valid action on polytabloids. In particular, recall that $S^{\lambda}$ is the Specht module, generated by standard polytabloids. We defined a dominance ordering on tableaux that was then extended to an ordering on polytabloids in Definition \ref{def: col dom poly}. 

Similarly to the previous section, our main goals will be to define a related $\mathcal{H}_n(0)$-action on $S^{\lambda}$ and a compatible composition series of submodules closed under that action: $$S^{\lambda}=M_1\supset M_2 \supset\cdots \supset M_{f^{\lambda}}\supset\emptyset.$$  As above, it turns out that selecting a total ordering on the basis $\{e_T\mid T\in\text{SYT}(\lambda)\}$ can easily give rise to both.  Using Definition \ref{def: col dom poly} above, consider the ordering $e_{T_1}\prec e_{T_2} \prec\cdots\prec e_{T_{f^{\lambda}}}$.  Then we can easily define the composition series: 
$$\text{span}(e_{T_1}, \dots, e_{T_{f^{\lambda}}})\supset \text{span}(e_{T_2}, \dots, e_{T_{f^{\lambda}}})\supset \cdots \supset \text{span}(e_{T_{f^{\lambda}}})\supset \emptyset.$$
Next, to ensure a compatible $\mathcal{H}_n(0)$ action to the composition series, it must be the case that if $$\pibar_i(e_{T_j})=\sum\limits_{k=1}^{f^\lambda}c_k e_{T_k},$$ then $c_k=0$ when $k>j$. Moreover, $c_j$ must be either 0 or -1, since $\pibar_i^2=-\pibar_i$.

To formalize the action, we will need the concepts of the \textit{leading term} and \textit{trailing term} with respect to a given basis and ordering.

\begin{definition}
    Assume $\mathcal{B}=\{v_1, \dots, v_l\}$ is a basis of a vector space, given in decreasing order according to $\leq$ defined on $\mathcal{B}$.
    
    Let the leading term of a vector $v$ be $\text{lt}(v)=\text{lt}(\sum c_iv_i)=c_jv_j$, where $j$ is the smallest index such that $c_j\neq0$.
    
    Let the trailing term of a vector $v$ be $\text{tt}(v)=\text{tt}\left(\sum c_iv_i\right)=c_jv_j$, where $j$ is the largest index such that $c_j\neq0$.
\end{definition}

Now, we recall the $\mathfrak{S}_n$-action on polytabloids, which we will deform to define the compatible $\mathcal{H}_n(0)$-action.

\begin{recall}
    $$s_i(e_T)=\begin{cases}
-e_T & i \hspace{2mm}\text{and $i+1$ in the same column} \\
e_T\pm \sum\limits_{T\prec S} c_{T,S}^ie_{S} & i \hspace{2mm}\text{and $i+1$ in the same row}\\
e_{s_iT} & i \hspace{2mm}\text{and $i+1$ in a different row and column}
\end{cases},$$ where $c_{T,S}^i\in\mathbb{C}$.
\end{recall}

\begin{definition}\label{def: specht hecke}
        $$\overline{\pi}_{i}e_T = \begin{cases}
        0 & \text{tt}(s_ie_T)=e_T\\
-e_T & e_{T}\prec e_{s_iT} \text{ or } s_ie_T=-e_T \\
e_{s_iT} & e_{s_iT}\prec e_T
\end{cases}$$
\end{definition}

\noindent Note that this $\pibar_i$ action is more involved than that on tabloids. This is because $s_i(e_T)$ may be a linear combination of standard polytabloids, when written in terms of the basis. Thus, we must account for additional cases when comparing via $\prec$. However, it is straightforward to see that, much like tabloids, there is an easier, equivalent action for which it will be more straightforward to show that the commuting relations hold.

\begin{definition}
        $$\overline{\pi}_{i}e_T = \begin{cases}
        0 & i,i+1\text{ in the same row of } T\\
-e_T & i \text{ southeast of } i+1 \text{ in } T \\
e_{s_iT} & i \text{ northwest of } i+1 \text{ in } T
\end{cases}$$
\end{definition}

\noindent As with the actions on $\widehat{M^{\lambda}}$, it is relatively straightforward to see that these two $\mathcal{H}_n(0)$-actions on $\widehat{S^{\lambda}}$ are equivalent.

\begin{proposition}
    $\widehat{S^{\lambda}}$ is a $\mathcal{H}_n(0)$-module under the $\pibar_i$ action defined above.
\end{proposition}
\begin{proof}
    Since both definitions are equivalent, we can use the second action to show that the commuting relations hold, and thus $\widehat{S^{\lambda}}$ is a $\mathcal{H}_n(0)$-module. The proofs for each of the commuting relations are analogous to those of Propositions \ref{prop: m squared}, \ref{prop: m length 2}, and \ref{prop: m length 3}.
\end{proof}

\noindent As with $\widehat{M^{\lambda}}$, we see that the ordering $\prec$ gives a natural way to construct a composition series of $\widehat{S^{\lambda}}$ such that each quotient space is a one-dimensional irreducible $\mathcal{H}_n(0)$-module that is isomorphic to some $C_I$. Recall that $I=\{i\mid \pibar_ie_T=-e_T\}$, so we monitor where the $\pibar_i$ action results in a negative loop. 

\begin{theorem}\label{thm: specht qsym char}
Under the quasisymmetric Frobenius characteristic map, we have $$F_{char}^Q(\widehat{S^{\lambda}})=s_{\lambda}=F_{char}(S^{\lambda})$$
\end{theorem}
\begin{proof}
First, order the $f^{\lambda}$ standard polytabloids that form the basis of $\widehat{S^{\lambda}}$ according to the total ordering $\prec$, from greatest to least. Call $e_{T_j}$ the $j^{\text{th}}$ polytabloid in the basis ordering. To construct the composition series, $\widehat{S^{\lambda}}=M_1\supset M_2\supset\cdots M_{f^{\lambda}}=\emptyset$, let $M_j=\text{span}{(e_{T_j}, e_{T_{j+1}}, \dots, e_{T_{f^{\lambda}}})}$. 

We note a couple of facts about this composition series. First, we observe that $\faktor{M_j}{M_{j+1}}$ is one-dimensional, generated by $[e_{T_j}]_j$. Furthermore, this composition series is, in fact, a composition series of submodules, and we do have $M_j\supset M_{j+1}$, as our action is based on $\prec$.

Now, $\faktor{M_j}{M_{j+1}}\cong C_I$, where $I=\{i\mid \pibar_ie_{T_j}=-e_{T_j}\}$. So, by the definition of the quasisymmetric Frobenius characteristic map, we have $$F_{char}^Q(\widehat{S^{\lambda}})=\sum\limits_{T\in\text{SYT}(\lambda)} F_{\{i\mid \pibar_ie_T=-e_T\}}.$$

Using the definition of the action, we have 
\begin{align*}
    F_{char}^Q(\widehat{S^{\lambda}})&= \sum\limits_{T\in\text{SYT}(\lambda)} F_{\{i\mid \pibar_ie_T=-e_T\}} \\
    &= \sum\limits_{T\in\text{SYT}(\lambda)} F_{\{i\mid i \text{ southeast of } i+1 \text{ in } T\}} \\
    &= \sum\limits_{T\in\text{SYT}(\lambda)} F_{\text{ides}(T)} \\
    &= s_{\lambda}
\end{align*}
\end{proof}

\section{Implications of the previous sections}\label{section: broad impl}
\noindent The proofs of Theorems \ref{thm: quasi char of m} and \ref{thm: specht qsym char} follow the same general approach. 
In particular, we created an action and a composition series according to the following steps:
\begin{enumerate}
    \item Define an ordering $\leq$ on a basis $\mathcal{B}$ of the $\mathfrak{S}_n$-module $M$.
    \item Using $\mathcal{B}$ and $\leq$, create a maximal decreasing sequence of nested subspaces: $$M=M_1\supset M_2\supset\cdots\supset M_k=\emptyset.$$
    \item Create a compatible $\mathcal{H}_n(0)$-action based on $\leq$ that causes the subspaces $M_i$ to be closed under that action, so that the sequence $$M=M_1\supset M_2\supset\cdots\supset M_k=\emptyset$$ forms a composition series of $M$.
\end{enumerate}

\noindent This approach suggests the following definition:

\begin{definition}\label{def: sqcc}
    Let $M$ be a $\mathfrak{S}_n$-module with basis $\mathcal{B}=\{v_1, \dots, v_k\}$ and a total ordering $\leq$ on $\mathcal{B}$ such that $v_1\leq v_2 \leq \cdots \leq v_k$. On the elements of $\mathcal{B}$, let $$\overline{\pi}_i(v_t)=\begin{cases}
            0 & \text{tt}(s_i(v_t))=v_t \\
            -v_t & s_i(v_t)=v_j> v_t\text{ or } s_i(v_t)=-v_t \\
            s_i(v_t) & s_i(v_t)=v_j< v_t
        \end{cases}.$$
        In particular, assume only these three cases appear on the basis $\mathcal{B}$.
    When $\pibar_i$ is a valid $\mathcal{H}_n(0)$-action on the vector space of $M$, call the resulting module $\widehat{M}$.  Then we must have a resulting composition series 
    $$\text{span}(v_1, \cdots, v_k)\supset \text{span}(v_2, \cdots, v_k)\supset \cdots \supset \text{span}(v_k)\supset \emptyset$$
    as in particular each of the vector spaces are closed under the action.
    
    We say $(M,\mathcal{B},\leq)$ is \textbf{strongly quasisymmetric characteristic compatible} (SQCC) if $$F_{char}^Q(\widehat{M})=F_{char}(M).$$
\end{definition}

\begin{theorem}
    $(M^{\lambda}, \{\{t\}\mid t\text{ is a row-strict tableau of shape } \lambda\}, \prec)$ is strongly quasisymmetric characteristic compatible.
\end{theorem}
\begin{proof}
    This follows from Theorem \ref{thm: quasi char of m} and Definition \ref{def: sqcc}.
\end{proof}

\begin{theorem}
    $(S^{\lambda}, \{e_T\mid T\in\text{SYT}(\lambda)\}, \prec)$ is strongly quasisymmetric characteristic compatible.
\end{theorem}
\begin{proof}
    This follows from Theorem \ref{thm: specht qsym char} and Definition \ref{def: sqcc}.
\end{proof}

\begin{proposition}\label{prop: all sqcc}
    For every $\mathfrak{S}_n$-module $M$, there exists a basis $\mathcal{B}_M$ and a total ordering $\leq_M$ such that $(M, \mathcal{B}_M, \leq_M)$ is strongly quasisymmetric characteristic compatible.
\end{proposition}
\begin{proof}
        If $M$ is an $\mathfrak{S}_n$-module, then we can write $M\cong\oplus c_{\lambda}S^{\lambda}$.  Consider the ordering $\prec$ on the basis of $S^\lambda$ defined in Definition \ref{def: col dom poly}. Using the isomorphism, create a compatible basis $\mathcal{B}_M$ for $M$ from the bases $\{e_T|T\in \text{SYT}(\lambda)\}$ within an isomorphic copy of $S^\lambda$ such that within each submodule, the relative order of the basis elements in Definition \ref{def: col dom poly} are preserved.  Call this new order on the basis elements of $\leq_M$.     When $\widehat{M}$ is a $\mathcal{H}_n(0)$-module, we also have $\widehat{M}\cong\oplus c_{\lambda} \widehat{S^{\lambda}}$ as vector spaces. Then, we have the following: 
        \begin{align*}
            F_{char}^Q(\widehat{M}) &= F_{char}^Q\left(\oplus c_{\lambda}\widehat{S^{\lambda}}\right) \\
            &=\sum c_{\lambda}F_{char}^Q\left(\widehat{S^{\lambda}}\right) \\
            &= \sum c_{\lambda} F_{char}\left(S^{\lambda}\right) \\
            &=\sum c_{\lambda}s_{\lambda} \\
            &= F_{char}\left(\oplus c_{\lambda}S^{\lambda}\right) \\
            &=F_{char}(M)
        \end{align*}
\end{proof}

\begin{corollary}\label{cor: commuting diagram}
    Given an $\mathfrak{S}_n$-module $M$, let $\mathcal{B}_M$ and total ordering $\leq_M$ be as constructed in Proposition \ref{prop: all sqcc}. Let $f$ be a map defined on the representations of $\mathfrak{S}_n$ such that $f(M)=\widehat{M}$, and $\iota:\text{Sym}\rightarrow\text{QSym}$ denote the inclusion map. Then, the following diagram commutes:
\[
\begin{tikzcd} \mathcal{G}_0(\mathfrak{S}_n) \arrow[r, "F_{char}"] \arrow[d, dashed, "f"] & \text{Sym} \arrow[d, hookrightarrow, "\iota"] \\
    \mathcal{G}_0(\mathcal{H}_n(0)) \arrow[r, "F_{char}^Q"]             & \text{QSym}
\end{tikzcd}
\]
\end{corollary}

This proposition and subsequent corollary imply that for every action on an $\mathfrak{S}_n$-module there exists a naturual deformed to a $\mathcal{H}_n(0)$-action with the same Frobenius image.
However, the specific deformation mentioned above may not always be easy to compute, due to the complexities of defining a 1-1 map between $S^{\lambda}$ and $M$. 

\section{The Coinvariant Algebra as a $\mathcal{H}_n(0)$-module}\label{section: coinv}

In this section, we are inspired by a broader motivation:
\begin{motivation} 
    Study the symmetric Frobenius image of $M$ by computing the quasisymmetric Frobenius image of $\widehat{M}$.
\end{motivation}  
 
\noindent Corollary \ref{cor: commuting diagram}, alas, does not address this motivation directly, since $\mathcal{B}_M$ and $\leq_M$ may not be natural on a generic module $M$. In fact, as the next example suggests, the action presented in Definition \ref{def: sqcc} may be too limiting if this is our goal, but a weaker condition may be more realistic. 

For our final example, we transition to defining a compatible $\mathcal{H}_n(0)$-action on a well-studied quotient space: the coinvariant algebra $\mathcal{R}_n$.  Before giving the formal definition, we briefly give some necessary background on Gr{\"o}bner bases in Section \ref{subsect: grobner basis} before returning to the example at hand in section \ref{subsect: coinvariant}.  

Let $\mathbb{C}[x_1, \dots, x_n]$ be the space of all functions in $n$ variables. Given a weak composition $\alpha\vDash_w d$, $x^{\alpha}\in\mathbb{C}[x_1, \dots, x_n]$ is the monomial $x_1^{\alpha_1}x_2^{\alpha_2}\cdots x_n^{\alpha_n}$, and we say that $x^{\alpha}$ has degree $d$.

\subsection{Gr{\"o}bner Bases}\label{subsect: grobner basis}

This subsection serves as a brief overview of Gr{\"o}bner bases and related results. We adopt the notation of Dummit and Foote~\cite{dummitFoote}, which gives a good introduction to the 1423ic for the unfamiliar reader. Following their notation, when the underlying ideal $I$ is obvious, for any $p\in \mathbb{C}[x_1, \dots, x_n]$, we use  $\overline{p}$ for the coset in $\faktor{\mathbb{C}[x_1,\dots,x_n]}{I}$ corresponding to $p$.

One difficulty with computations in quotient spaces is that elements of quotient spaces have multiple coset representatives and thus it may be difficult to determine if $\overline{p}$ equals $\overline{q}$ for two distinct elements $p,q \in \mathbb{C}[x_1, \dots, x_n]$.  In the case that $I$ is a nonzero ideal of $\mathbb{C}[x_1, \dots, x_n]$, Gr{\"o}bner bases are employed as a tool to overcome this obstacle when doing computations in  $Q=\faktor{\mathbb{C}[x_1, \dots, x_n]}{I}$.  In such a setting, we must first select a monomial ordering on $\mathbb{C}[x_1, \dots, x_n]$:

\begin{definition}
    A \textbf{monomial ordering} on $\mathbb{C}[x_1, \dots, x_n]$ is a total ordering such that, if $x^{\alpha}\leq x^{\beta}$ and $x^{\gamma}$ is any other monomial, then $x^{\alpha}x^{\gamma}\leq x^{\beta}x^{\gamma}$. 
\end{definition}

The lexicographic monomial ordering, explicitly defined below, is a widely used monomial ordering on $\mathbb{C}[x_1, \dots, x_n]$.

\begin{definition}\label{def: lex ordering}
    Given two monomials $x^\alpha$ and $x^{\beta}$, we say $x^{\alpha}< x^{\beta}$ in \textbf{lexicographic monomial ordering} if the first nonzero entry in the sequence $\beta-\alpha$ is strictly positive.
\end{definition}

With respect to the monomial ordering, we then have a ``first" monomial that occurs with nonzero coefficient in every polynomial, leading to two related definitions:
\begin{definition}
    Given a polynomial $p\in\mathbb{C}[x_1, \dots, x_n]$, let $$p=c_1x^{\alpha_1}+c_2x^{\alpha_2}+\cdots+c_mx^{\alpha_m}$$ with $x^{\alpha_1}>x^{\alpha_2}>\cdots>x^{\alpha_m}$ and $c_i\neq 0$ for all $i$. Then, the \textbf{leading term} of $p$, denoted $\text{lt}(p)$, is $c_1x^{\alpha_1}$ and the \textbf{leading monomial} of $p$, denoted $\text{lm}(p)$, is $x^{\alpha_1}$.
\end{definition}

Note that as defined monomial $\text{lt}(p)$ may have any nonzero coefficient, while the monomial $\text{lm}(p)$ will have a coefficient of 1.  

Gr{\"o}bner bases, misleadingly named as they are not actually bases of a vector space, are defined using ideals generated by leading terms of polynomials.

\begin{definition}
    Given a nonzero ideal $I$ in $\mathbb{C}[x_1, \dots, x_n]$, let $\text{lt}(I)$ be the ideal generated by the leading terms of polynomials in $I$:
    $$\text{lt}(I)=\langle \text{lt}(p)\mid p\in I\rangle$$
\end{definition}

\begin{definition}
    A \textbf{Gr{\"o}bner basis} $G=\{g_1, \dots, g_m\}$ of a nonzero ideal $I$ in the ring $\mathbb{C}[x_1, \dots, x_n]$ is a finite generating set of $I$ whose leading terms generate $\text{lt}(I)$. Thus, $$I=\langle g_1, \dots, g_m\rangle \text{ and } \text{lt}(I)=\langle \text{lt}(g_1), \dots, \text{lt}(g_m)\rangle.$$
\end{definition}

\begin{definition}
    A \textbf{reduced Gr{\"o}bner basis} $G=\{g_1, \dots, g_m\}$ of nonzero ideal $I$ in the ring $\mathbb{C}[x_1,\dots, x_n]$ is a Gr{\"o}bner basis where 
    \begin{itemize}
        \item $\text{lt}(g_i)$ has coefficient 1 for each $i$
        \item no monomial appearing in $g_j$ is divisible by $\text{lt}(g_i)$ when $i\neq j$. 
    \end{itemize}
\end{definition}

\begin{lemma}[Theorem 27 in \cite{dummitFoote}]
    Given a monomial ordering on $\mathbb{C}[x_1, \dots, x_n]$, there is a unique reduced Gr{\"o}bner basis for every nonzero ideal $I$ in $\mathbb{C}[x_1,\dots, x_n]$. 
\end{lemma}

By using a reduced Gr{\"o}bner basis $G=\{g_1,\dots, g_m\}$, we are able to define a division algorithm on $\mathbb{C}[x_1, \dots, x_n]$ that results in a unique remainder. Reminiscent of the Euclidean algorithm, if $p\in\mathbb{C}[x_1, \dots, x_n]$, then:
\begin{itemize}
    \item if $\text{lt}(p)$ is divisible by $\text{lt}(g_i)$ such that $\text{lt}(p)=a_i\text{lt}(g_i)$, add $a_i$ to the ``quotient" $q_i$, and replace $p$ by $p-a_ig_i$. Repeat this process.
    \item if $\text{lt}(p)$ is not divisible by $\text{lt}(g_i)$ for any $i$, add the leading term of $p$ to the ``remainder" $r$, and replace $p$ by $p-\text{lt}(p)$. Repeat this process.
\end{itemize}
This algorithm does terminate, resulting in $$p=q_1g_1+q_2g_2+\cdots+q_mg_m+r.$$

Moreover, when working only with the reduced Gr{\"o}bner basis $G$, the remainder of any element in $\mathbb{C}[x_1, \dots, x_n]$ depends only on $I$ and the choice of $\leq$. Otherwise, the remainder also depends on the choice of Gr{\"o}bner basis $G$.

\begin{lemma}[Theorem 23 in \cite{dummitFoote}]\label{lem: unique remainder}
    Let $G=\{g_1, \dots, g_m\}$ be a Gr{\"o}bner basis for a nonzero ideal $I$ in $\mathbb{C}[x_1, \dots, x_n]$. Then, every polynomial $p\in\mathbb{C}[x_1, \dots, x_n]$ can be written uniquely as $$p=f_I+r,$$ where $f_I\in I$ and no nonzero monomial term in $r$ is divisible by the leading terms of the polynomials in $G$. Moreover, the remainder $r$ is a unique representative in the coset of $p$ in $\faktor{\mathbb{C}[x_1, \dots, x_n]}{I}$.
\end{lemma}

\begin{definition}
    Let $G$ be a Gr{\"o}bner basis of nonzero ideal $I$ and $p\in\mathbb{C}[x_1,\dots,x_n]$, and write $p=f_I+r$ as given in Lemma \ref{lem: unique remainder}. If $f_I$ is a nonzero polynomial, then we say that $p$ is \textbf{reducible} by $G$. If $f_I=0$, then $p$ is \textbf{reduced} with respect to $G$.
\end{definition}

\begin{corollary}
    Let $I$ be a nonzero ideal in $\mathbb{C}[x_1, \dots, x_n]$ with Gr{\"o}bner basis $G$. Then, every coset $\overline{p}\in \faktor{\mathbb{C}[x_1, \dots, x_n]}{I}$ has a unique coset representative $\widetilde{p}\in\mathbb{C}[x_1, \dots, x_n]$ that is reduced with respect to $G$.
\end{corollary}
\begin{proof}
    By Lemma \ref{lem: unique remainder}, to find the unique representative, we need only apply the division algorithm to any coset representative.  
\end{proof}

Going forward, when the underlying ideal $I$ and ordering $\leq$ is obvious, we consider the unique reduced Gr{\"o}bner basis $G$. Then, we use $\widetilde{p}\in\mathbb{C}[x_1, \dots, x_n]$ to denote this unique coset representative of $\overline{p}$ that is reduced with respect to $G$.

\subsection{The Coinvariant Algebra $\mathcal{R}_n$}\label{subsect: coinvariant}

Next, we turn our attention to a well-studied action of $\mathfrak{S}_n$ and a compatible quotient space. Define an $\mathfrak{S}_n$-action on $\mathbb{C}[x_1, \dots, x_n]$ by $$s_i(x^{\alpha})=x^{s_i(\alpha)},$$ where $s_i(\alpha)=s_i((\alpha_1, \dots, \alpha_i, \alpha_{i+1}, \dots, \alpha_l))=(\alpha_1, \dots, \alpha_{i+1}, \alpha_{i}, \dots, \alpha_l)$.

\begin{definition} Let $S$ be the ideal $\langle e_1(x_1, \dots, x_n), \dots, e_n(x_1, \dots, x_n)\rangle$.  Then the \textbf{coinvariant algebra} $\mathcal{R}_n$ is the quotient space $\faktor{\mathbb{C}[x_1\dots, x_n]}{S}$. 
\end{definition}

Now, the ideal $S$ is invariant under the aforementioned $\mathfrak{S}_n$-action and thus one can consider the action on the quotient $\mathcal{R}_n$. In particular, $$s_i(\overline{x^{\alpha}})=\overline{x^{s_i(\alpha)}}$$ is thus a well-defined action on $\mathcal{R}_n$ that does not depend on the choice of coset representative.

For additional details on $\mathcal{R}_n$ and other well-known quotient spaces in this area, we recommend the unfamiliar reader see the excellent text by Bergeron \cite{bergeron}.

As with the previous $\mathcal{H}_n(0)$-modules studied here, an ordering will be used throughout this section. Here, we consider the lexicographic ordering on monomials defined in the previous subsection.

\begin{lemma}[See \cite{bergeron}]\label{lem: grobner basis}
With respect to lexicographic ordering $\leq$, the ideal $S$ has reduced Gr{\"o}bner basis $$G=\{g_i=h_i(x_i, \dots, x_n) : 1\leq i\leq n\}.$$
\end{lemma}

Artin \cite{artin} showed that the set of ``sub-staircase" monomials form a basis of $\mathcal{R}_n$. 

\begin{definition}\label{def: artin basis}
    Let $A_n=\{\alpha=(\alpha_1, \alpha_2, \dots, \alpha_n)\mid 0\leq\alpha_j< j\}$.
    Then, define the set of \textbf{sub-staircase monomials} as $$\mathcal{A}_n=\{\overline{x^{\alpha}}\mid \alpha\in A_n\}.$$
\end{definition}

\begin{lemma}\label{lem: basis of coinv}
$\mathcal{A}_n$ is a basis for $\mathcal{R}_n$, and $\text{dim}(\mathcal{R}_n)=n!$.
\end{lemma}

We will use lexicographic ordering to define a $\mathcal{H}_n(0)$-action on $\mathcal{R}_n$. To do so, we will first need to analyze $\mathcal{R}_n$ as an $\mathfrak{S}_n$-module.

\begin{proposition}\label{prop: lt results}
Assume $\alpha\in A_n$. If $s_i(\alpha)\notin A_n$, then it must be that $\text{lm}(\widetilde{x^{s_i(\alpha)}})\geq x^{\alpha}$. If equality holds, then $\text{lt}(\widetilde{x^{s_i(\alpha)}})=-x^{\alpha}$.
\end{proposition}
\begin{proof}
We write $$x^{\alpha}=x_1^{\alpha_1}\cdots x_i^{\alpha_i}x_{i+1}^{\alpha_{i+1}}\cdots x_n^{\alpha_n}$$ and $$x^{s_i(\alpha)}=x_1^{\alpha_1}\cdots x_i^{\alpha_{i+1}}x_{i+1}^{\alpha_{i}}\cdots x_n^{\alpha_n}.$$
\setcounter{case}{0}
Since we assume that $x^{s_i(\alpha)}$ is reducible by $G$ and $\alpha\in A_n$, it must be that $\alpha_{i+1}=i$ as a result of Lemma \ref{lem: basis of coinv}. Thus, we can rewrite $$x^{s_i(\alpha)}=x_1^{\alpha_1}\cdots x_i^ix_{i+1}^{\alpha_{i}}\cdots x_n^{\alpha_n}.$$

Now, in this way, it is easy to see that $x^{s_i(\alpha)}$ is reducible by $g_i$ specifically. In particular, the leading term of $g_i$ is $x_i^i$ in lexicographic order. As such, we may reduce $x^{s_i(\alpha)}$ by $g_i$, resulting in a new polynomial $p(x_1,\dots,x_n)$ where 
$$\text{lt}(p(x_1,\dots,x_n))=m(x_1, \dots, x_n)=-x_1^{\alpha_1}\cdots x_{i-1}^{\alpha_{i-1}}x_i^{i-1}x_{i+1}^{\alpha_{i}+1}x_{i+2}^{\alpha_{i+2}}\cdots x_n^{\alpha_n}.$$ 
We aim to show that, when fully reduced, $m(x_1,\dots,x_n)$ remains the leading term of $\widetilde{x^{s_i(\alpha)}}$ after successive reductions.

For further reduction, we note there may be monomials of the form $x_1^{\alpha_1}\cdots x_{k}^{\alpha_k+r}\cdots x_n^{\alpha_n}$, $k>i$ and $r>0$, which must be reduced by $g_k$. Since $g_k=h_k(x_k, \dots, x_n)$, we note that with the reduction, any new terms continue to have the same prefix $x_1^{\alpha_1}\cdots x_{i-1}^{\alpha_{i-1}}x_i^{i-1}$. Thus, 
$$\text{lt}(\widetilde{x^{s_i(\alpha)}})=m(x_1,\dots,x_n).$$

If $\alpha_i=i-1$, then $\alpha_i+1=i=\alpha_{i+1}$. Thus, $m(x_1,\dots, x_n)=-x^{\alpha}$. Otherwise, if $\alpha_i<i-1$, we have $\text{lm}(\widetilde{x^{s_i(\alpha)}})>x^{\alpha}$.
\end{proof}

It is a result of Chevalley \cite{chevalley} that, with this $\mathfrak{S}_n$-action, we have the left regular representation, so that $$F_{char}(\mathcal{R}_n)=\sum\limits_{\lambda\vdash n} f^{\lambda}s_{\lambda}.$$

If we view $\mathcal{R}_n$ as a \textit{graded} $\mathfrak{S}_n$-module, we have from Lusztig, formalized by Stanley in \cite{grFrob}, that $$F_{char}(\mathcal{R}_n;q)=\sum\limits_{\lambda\vdash n}\sum\limits_{T\in\text{SYT}(\lambda)} q^{\text{maj}(T)}s_{\lambda}.$$ We will notice that, at $q=1$, we obtain $F_{char}(\mathcal{R}_n)$ as an ungraded $\mathfrak{S}_n$-module. As such, we will work with the \textit{graded} quasisymmetric Frobenius characteristic of $\mathcal{R}_n$, where, for monomial rings, we have
$$F_{char}^Q(\langle x^{\alpha}\rangle;q)=q^{\text{deg}(x^{\alpha})}F_{\{i\mid \pibar_i (x^{\alpha})=-x^{\alpha}\}}.$$

Next, in the style of how we defined the $\mathcal{H}_n(0)$-action on the Specht modules, let $\widehat{\mathcal{R}_n}$ be the vector space $\mathcal{R}_n$ with the goal of defining a valid $\mathcal{H}_n(0)$-action. In particular, we define an $\mathcal{H}_n(0)$-action on $\widehat{\mathcal{R}_n}$. In particular, we assume $\alpha\in A_n$ so that the action is defined on a basis for $\mathcal{R}_n$ and can be extended linearly:

\[
\pibar_i(\overline{x^{\alpha}})=
\begin{cases}
        0 &\text{if } x^{s_i\alpha}=x^{\alpha} \\    
        -\overline{x^{\alpha}} &\text{if } \text{lm}(\widetilde{x^{s_i\alpha}})> x^{\alpha} \text{ or } \text{lt}(\widetilde{x^{s_i\alpha}})=-x^{\alpha} \\
        \overline{x^{s_i(\alpha)}} &\text{if } x^{s_i\alpha}<x^{\alpha}
    \end{cases}
\]

By Proposition \ref{prop: lt results}, these are the only cases to consider on $\widehat{\mathcal{R}_n}$. By this same proposition, we know that if $x^{\alpha}>\text{lm}(\widetilde{x^{s_i(\alpha)}})$, then $x^{s_i(\alpha)}$ is not reducible by $G$, which is why the leading term notation is omitted in the third case above. If $x^{s_i(\alpha)}<x^{\alpha}$, then we must have $\alpha_{i}>\alpha_{i+1}$. Similarly, if $x^{s_i(\alpha)}=x^{\alpha}$, then $\alpha_i=\alpha_{i+1}$. This tells us that we can rewrite our $\mathcal{H}_n(0)$-action on $\widehat{\mathcal{R}_n}$ as follows, again assuming that $\alpha\in A_n$:

\[
    \pibar_i(\overline{x^{\alpha}}) = 
    \begin{cases}
        0 &\text{if } \alpha_i=\alpha_{i+1} \\
        -\overline{x^{\alpha}} &\text{if } \alpha_i<\alpha_{i+1} \\    
        \overline{x^{s_i(\alpha)}} &\text{if } \alpha_i>\alpha_{i+1}
    \end{cases}
\]

Now, we'll show that $\widehat{\mathcal{R}_n}$ is a valid $\mathcal{H}_n(0)$-module using this easier action. The proof is by contradiction. To reduce the number of cases, we will assume that $\overline{x^{\alpha}}$ is a counterexample to the length three commuting relations at positions $j, j+1, j+2$, and show that in fact this counterexample must occur at positions 3, 4, 5 (or earlier).

\begin{lemma}\label{lem: length 3 counter}
Let $j, \overline{x^{\alpha}}$ be a counterexample to the length three commuting relations. In particular, suppose $$\pibar_j\pibar_{j+1}\pibar_j(\overline{x^{\alpha}})\neq\pibar_{j+1}\pibar_j\pibar_{j+1}(\overline{x^{\alpha}}).$$
Then, there exists a non-distinct list $\beta_3, \beta_4, \beta_5\in\{0,1,2\}$ such that $$\pibar_3\pibar_4\pibar_3(\overline{x_3^{\beta_3}x_4^{\beta_4}x_5^{\beta_5}})\neq\pibar_4\pibar_3\pibar_4(\overline{x_3^{\beta_3}x_4^{\beta_4}x_5^{\beta_5}}).$$
\end{lemma}
\begin{proof}

Note that the action of $\pibar_i$ on a monomial is dependent only on the relative order of the exponents themselves, and not the values of the exponents. Thus, $\pibar_j\pibar_{j+1}\pibar_j$ and $\pibar_{j+1}\pibar_j\pibar_{j+1}$ act on $\overline{x^{\alpha}}$, giving either 0 or a permutation of the exponents in positions $j, j+1$, and $j+2$, which depends only on the relative sizes of $\alpha_j, \alpha_{j+1}$, and $\alpha_{j+2}$. Thus, if  $$\pibar_j\pibar_{j+1}\pibar_j(\overline{x^{\alpha}})\neq\pibar_{j+1}\pibar_j\pibar_{j+1}(\overline{x^{\alpha}}),$$ then it must be that if $\beta_3, \beta_4$ and $\beta_5$ are chosen to have the same relative order amongst themselves that $\alpha_j$, $\alpha_{j+1}$, and $\alpha_{j+2}$ have, acting by $\pibar_3\pibar_4\pibar_3$ and $\pibar_4\pibar_3\pibar_4$ will result in either 0 or an analogous permutation of $\beta_3, \beta_4$, and $\beta_5$ on the exponents of $\overline{x^{\beta}}$, so $$\pibar_3\pibar_4\pibar_3(\overline{x_3^{\beta_3}x_4^{\beta_4}x_5^{\beta_5}})\neq\pibar_4\pibar_3\pibar_4(\overline{x_3^{\beta_3}x_4^{\beta_4}x_5^{\beta_5}}).$$ Thus, we need only ensure that all relative orders of a triple $(\alpha_j, \alpha_{j+1}, \alpha_{j+2})$ may be achieved in the 3rd, 4th, and 5th exponents. Since the possible exponents of $x_3, x_4$, and $x_5$ include 0, 1, and 2, any relative order of three exponents can be achieved.
\end{proof}

\begin{example}
    For $n=9$, let $x^{\alpha}=x_7^5x_8^2x_9^6$. Then, we choose $x^{\beta}=x_3^1x_4^0x_5^2$, since $\alpha_8<\alpha_7<\alpha_9$ and $\beta_4<\beta_3<\beta_5$, where the indices are now shifted downwards by 4. Note that $\pibar_7\pibar_8\pibar_7(\overline{x^{\alpha}})=\overline{x_7^2x_8^5x_9^6}=s_7(\overline{x^{\alpha}})$ and $\pibar_3\pibar_4\pibar_3(\overline{x^{\beta}})=\overline{x_3^0x_4^1x_5^2}=s_{(7-4)}(\overline{x^{\beta}})$. Thus, if $\overline{x^{\alpha}}$ were a counterexample to the commuting relations, then $\overline{x^{\beta}}$ would be as well.
\end{example}

\begin{lemma}\label{lem: r length 3}
With the $\pibar_i$ action defined on $\widehat{\mathcal{R}_n}$, let $\alpha\in A_n$. Then, we have $$\pibar_3\pibar_4\pibar_3(\overline{x_1^{\alpha_1}\cdots x_3^{\alpha_3}x_4^{\alpha_4}x_5^{\alpha_5}\cdots x_n^{\alpha_n}})=\pibar_4\pibar_3\pibar_4(\overline{x_1^{\alpha_1}\cdots x_3^{\alpha_3}x_4^{\alpha_4}x_5^{\alpha_5}\cdots x_n^{\alpha_n}}).$$
\end{lemma}
\begin{proof}
The numerous cases, where we consider all relative orders of $\alpha_3, \alpha_4$, and $\alpha_5$, are similar in difficulty, so we give two specific cases here.

\begin{case}
    Suppose $\alpha_3<\alpha_5<\alpha_4$.
    \begin{align*}
        \pibar_3\pibar_4\pibar_3(\overline{x_1^{\alpha_1}\cdots x_3^{\alpha_3}x_4^{\alpha_4}x_5^{\alpha_5}\cdots x_n^{\alpha_n}}) &= -\pibar_3\pibar_4(\overline{x_1^{\alpha_1}\cdots x_3^{\alpha_3}x_4^{\alpha_4}x_5^{\alpha_5}\cdots x_n^{\alpha_n}}) \\
        &= -\pibar_3(\overline{x_1^{\alpha_1}\cdots x_3^{\alpha_3}x_4^{\alpha_5}x_5^{\alpha_4}\cdots x_n^{\alpha_n}}) \\
        &= \overline{x_1^{\alpha_1}\cdots x_3^{\alpha_3}x_4^{\alpha_5}x_5^{\alpha_4}\cdots x_n^{\alpha_n}} \\
        \pibar_4\pibar_3\pibar_4(\overline{x_1^{\alpha_1}\cdots x_3^{\alpha_3}x_4^{\alpha_4}x_5^{\alpha_5}\cdots x_n^{\alpha_n}}) &= \pibar_4\pibar_3(\overline{x_1^{\alpha_1}\cdots x_3^{\alpha_3}x_4^{\alpha_5}x_5^{\alpha_4}\cdots x_n^{\alpha_n}}) \\
        &= -\pibar_4(\overline{x_1^{\alpha_1}\cdots x_3^{\alpha_3}x_4^{\alpha_5}x_5^{\alpha_4}\cdots x_n^{\alpha_n}}) \\
        &= \overline{x_1^{\alpha_1}\cdots x_3^{\alpha_3}x_4^{\alpha_5}x_5^{\alpha_4}\cdots x_n^{\alpha_n}}
    \end{align*}
\end{case}

\begin{case}
    Now, assume $\alpha_4<\alpha_5=\alpha_3$.
    \begin{align*}
        \pibar_3\pibar_4\pibar_3(\overline{x_1^{\alpha_1}\cdots x_3^{\alpha_3}x_4^{\alpha_4}x_5^{\alpha_5}\cdots x_n^{\alpha_n}}) &= \pibar_3\pibar_4(\overline{x_1^{\alpha_1}\cdots x_3^{\alpha_4}x_4^{\alpha_3}x_5^{\alpha_5}\cdots x_n^{\alpha_n}}) \\
        &= \pibar_3(0) \\
        &= 0 \\
        \pibar_4\pibar_3\pibar_4(\overline{x_1^{\alpha_1}\cdots x_3^{\alpha_3}x_4^{\alpha_4}x_5^{\alpha_5}\cdots x_n^{\alpha_n}}) &= -\pibar_4\pibar_3(\overline{x_1^{\alpha_1}\cdots x_3^{\alpha_3}x_4^{\alpha_4}x_5^{\alpha_5}\cdots x_n^{\alpha_n}}) \\
        &= -\pibar_4(\overline{x_1^{\alpha_1}\cdots x_3^{\alpha_4}x_4^{\alpha_3}x_5^{\alpha_5}\cdots x_n^{\alpha_n}}) \\
        &= 0
    \end{align*}
\end{case}
\end{proof}

\begin{proposition}
$\widehat{\mathcal{R}_n}$ is a $\mathcal{H}_n(0)$-module under the previously defined $\pibar_i$ action.
\end{proposition}
\begin{proof}
To show that $\widehat{\mathcal{R}_n}$ is a $\mathcal{H}_n(0)$-action, we need only show that the commuting relations hold.

First, consider $\pibar_i^2(\overline{x_1^{\alpha_1}\cdots x_i^{\alpha_i}x_{i+1}^{\alpha_{i+1}}\cdots x_n^{\alpha_n}})$. There are three cases, and we work through each. 

\setcounter{case}{0}
\begin{case}
    Assume that $\alpha_i>\alpha_{i+1}$. Then, 
\begin{align*}
    \pibar_i^2(\overline{x_1^{\alpha_1}\cdots x_i^{\alpha_i}x_{i+1}^{\alpha_{i+1}}\cdots x_n^{\alpha_n}}) &= \pibar_i(\overline{x_1^{\alpha_1}\cdots x_i^{\alpha_{i+1}}x_{i+1}^{\alpha_{i}}\cdots x_n^{\alpha_n}}) \\
    &= -\overline{x_1^{\alpha_1}\cdots x_i^{\alpha_{i+1}}x_{i+1}^{\alpha_{i}}\cdots x_n^{\alpha_n}}) \\
    &= -\pibar_i(\overline{x_1^{\alpha_1}\cdots x_i^{\alpha_i}x_{i+1}^{\alpha_{i+1}}\cdots x_n^{\alpha_n}})
\end{align*}
\end{case}

\begin{case}
    Assume that $\alpha_i<\alpha_{i+1}$. Then, 
\begin{align*}
    \pibar_i^2(\overline{x_1^{\alpha_1}\cdots x_i^{\alpha_i}x_{i+1}^{\alpha_{i+1}}\cdots x_n^{\alpha_n}}) &= \pibar_i(-\overline{x_1^{\alpha_1}\cdots x_i^{\alpha_{i}}x_{i+1}^{\alpha_{i+1}}\cdots x_n^{\alpha_n}}) \\
    &= \overline{x_1^{\alpha_1}\cdots x_i^{\alpha_{i+1}}x_{i+1}^{\alpha_{i}}\cdots x_n^{\alpha_n}}) \\
    &= -\pibar_i(\overline{x_1^{\alpha_1}\cdots x_i^{\alpha_i}x_{i+1}^{\alpha_{i+1}}\cdots x_n^{\alpha_n}})
\end{align*}
\end{case}

\begin{case}
    Assume that $\alpha_i=\alpha_{i+1}$. Then, 
\begin{align*}
    \pibar_i^2(\overline{x_1^{\alpha_1}\cdots x_i^{\alpha_i}x_{i+1}^{\alpha_{i+1}}\cdots x_n^{\alpha_n}}) &= \pibar_i(0) \\
    &= 0 \\
    &= -\pibar_i(\overline{x_1^{\alpha_1}\cdots x_i^{\alpha_i}x_{i+1}^{\alpha_{i+1}}\cdots x_n^{\alpha_n}})
\end{align*}
\end{case}

Now, we consider $\pibar_i\pibar_j(\overline{x_1^{\alpha_1}\cdots x_n^{\alpha_n}})$, where $|i-j|\geq 2$. However, since $|i-j|\geq 2$, $\alpha_i, \alpha_{i+1}$ and $\alpha_j, \alpha_{j+1}$ have no direct impact on the actions of $\pibar_j$ and $\pibar_i$, respectively.

Finally, to show that $\pibar_i\pibar_{i+1}\pibar_i(\overline{x_1^{\alpha_1}\cdots x_n^{\alpha_n}})=\pibar_{i+1}\pibar_i\pibar_{i+1}(\overline{x_1^{\alpha_1}\cdots x_n^{\alpha_n}})$, we note that any counterexample to this commuting relation must also be a counterexample if $i=3$ by Lemma \ref{lem: length 3 counter}. Thus, we need only consider the case when $i=3$, but then we are done by Lemma \ref{lem: r length 3}.

Therefore, $\widehat{\mathcal{R}_n}$ is a $\mathcal{H}_n(0)$-module under the previously defined $\pibar_i$ action.
\end{proof}

Now, we are ready to calculate the quasisymmetric Frobenius characteristic of $\widehat{\mathcal{R}_n}$. Before we begin the proof, we need a definition and some facts about RSK. 

\begin{definition}\label{def: inv code}
    Define the inversion code of a permutation $w\in\mathfrak{S}_n$ to be the sequence $c(w)=(c_1, c_2, \dots, c_n)$ for $w\in\mathfrak{S}_n$ such that $c_i=|\{j<i\mid w^{-1}_j>w^{-1}_i\}|$ is the number of letters to the right of $i$ in the one-line notation of $w$ that are smaller than it. 
\end{definition}

This sequence is such that $0\leq c_i<i$ for all $i\leq n$. In particular, the inversion codes $c(w)$ are a variation of the well known Lehmer codes $L(w)$, with $$c(w)=\text{rev}(L(\text{flip}(\text{rev}(w^{-1}))))$$ where $\text{flip}(w)$ is the permutation defined by $\text{flip}(w)_i=n+1-w_i$. Thus, the inversion codes (and thus the permutations of  $\mathfrak{S}_n$) are in bijection with the sub-staircase monomials $\mathcal{A}_n$. 

Furthermore, we will also utilize the RSK correspondence and the following fact:

\begin{fact}(See \cite{sagan}.)
    Given a permutation $w\in\mathfrak{S}_n$, RSK assigns a pair of standard Young tableaux $(P,Q)$, where
    $$\text{ides}(P)=\text{ides}(w)$$ and $$\text{ides}(Q)=\text{des}(w).$$
\end{fact}
With this, we have our main result of this section.
\begin{theorem}\label{thm: graded coinvariants}
    $$F_{char}^Q(\widehat{\mathcal{R}_n};q)=\sum\limits_{\lambda\vdash n}\sum\limits_{Q\in\text{SYT}(\lambda)} q^{\text{maj}(Q)}s_{\lambda}=F_{char}(\mathcal{R}_n;q)$$
\end{theorem}

\begin{proof}
Assume that $\beta^j\in A_n$ for $1\leq j\leq n!$. Then, the monomials $x^{\beta^j}$ can be ordered lexicographically, from greatest to least such that $x^{\beta^1}$ is the largest monomial in the list. Then, let $\beta^j_k$ be the $k^{\text{th}}$ part of $\beta^j$.

We can define a composition series of our module $\widehat{\mathcal{R}_n}$ such that, if $M_j=\text{span},(\overline{x^{\beta^j}}, \overline{x^{\beta^{j+1}}}, \dots, \overline{x^{\beta^{n!}}})$, then
$$\widehat{\mathcal{R}_n}= M_1\supset M_2\supset\cdots M_{n!}=\emptyset.$$

We note a couple of facts about this composition series. First, using $[\cdot]_j$ to denote the equivalence classes in $\faktor{M_j}{M_{j+1}}$, we observe that $\faktor{M_j}{M_{j+1}}$ is one-dimensional, generated by $[\overline{x^{\beta^{j}}}]_j$. Furthermore, this composition series is, in fact, a composition series of submodules since we have $M_j\supset M_{j+1}$, as our action is based on $\leq$.

Since $\faktor{M_j}{M_{j+1}}=\langle [\overline{x^{\beta^{j}}}]_j\rangle\cong C_I$, where $I=\{i\mid \pibar_i(\overline{x^{\beta^{j}}})=-\overline{x^{\beta^{j}}}\}$. However, from the definition of the $\pibar_i$ action, this means $I=\{i\mid \beta^{j}_{i}<\beta^{j}_{i+1}\}$. Therefore,  
$$F_{char}^Q(\widehat{\mathcal{R}_n};q)=\sum\limits_{k=1}^{n!} q^{\text{deg}(x^{\beta^k})} F_{\{i\mid \beta^{k}_{i}<\beta^{k}_{i+1}\}}$$

Now, we know that the exponents $\beta^k$ are in bijection with permutations $w\in\mathfrak{S}_n$ with corresponding inversion code $c(w)$, where $c(w)$ counts the number of inversions as defined above. So, if $\beta^k_i<\beta^k_{i+1}$ for some $i$, then $c_i<c_{i+1}$ in $c(w)$. By definition of $c_i$, this means there are more letters to the right of $i+1$ that are smaller than it in $w$ than are to the right of $i$ that are smaller. In particular, this must mean that $i$ is to the right of $i+1$ in $w$. 
Conversely, if $i\in \text{ides}(w)$, $i+1$ is to the left of $i$ and every element smaller than $i$ and to the right of it is also smaller than $i+1$ and to it's right.  Since $i$ is also smaller than $i+1$ and to the right of $i+1$, we have $c_{i+1}>c_i$.
Thus, for a given $\beta^k$, we have $\{i\mid \beta^{k}_{i}<\beta^{k}_{i+1}\}=\text{ides}(w)$ for the unique $w\in\mathfrak{S}_n$ with $c(w)=\beta^k$.

Thus, we have the following:
\begin{align*}
    F_{char}^Q(\widehat{\mathcal{R}_n};q) &= \sum\limits_{k=1}^{n!} q^{\text{deg}(x^{\beta^k})} F_{\{i\mid \beta^{k}_{i}<\beta^{k}_{i+1}\}} \\
    &= \sum\limits_{w\in\mathfrak{S}_n} q^{\text{inv}(w)} F_{\text{ides}(w)}
\end{align*}

Next, we use the Foata bijection \cite{foata} to exploit the relationship between the inversion number and major index. In particular, we also note that the Foata bijection preserves the i-descent sets of the permutation. Thus, we have
    $$F_{char}^Q(\widehat{\mathcal{R}_n};q)=\sum\limits_{w\in\mathfrak{S}_n} q^{\text{maj}(w)} F_{\text{ides}(w)}$$

By the RSK correspondence on permutations in $\mathfrak{S}_n$, we have a bijection to pairs of standard Young tableaux, giving us 
\begin{align*}
    F_{char}^Q(\widehat{\mathcal{R}_n};q) &= \sum\limits_{w\in\mathfrak{S}_n} q^{\text{maj}(w)} F_{\text{ides}(w)} \\
    &= \sum\limits_{\lambda\vdash n} \sum\limits_{Q\in\text{SYT}(\lambda)} q^{\text{maj}(Q)}\sum\limits_{P\in\text{SYT}(\lambda)} F_{\text{ides}(P)} \\
    &= \sum\limits_{\lambda\vdash n} \sum\limits_{Q\in\text{SYT}(\lambda)} q^{\text{maj}(Q)}s_{\lambda}
\end{align*}
\end{proof}

\begin{corollary}
    $$F_{char}^Q(\widehat{\mathcal{R}_n})=\sum\limits_{\lambda\vdash n} f^{\lambda} s_{\lambda}=F_{char}(\mathcal{R}_n)$$
\end{corollary}
\begin{proof}
    Given Theorem \ref{thm: graded coinvariants}, set $q=1$. The result follows, recalling that $f^{\lambda}$ is the number of standard Young tableaux of shape $\lambda$.
\end{proof}

\begin{remark}
Throughout this section, we recall that we used lexicographic ordering to define an ordering on $\mathcal{R}_n$. However, we could choose any monomial ordering for which we have $x_i^c>x_{i+1}^c$ for all $i$ and $c$ and get the same resulting quasisymmetric Frobenius image. For example, this result holds for $\leq$ defined based on degree reverse lexicographic ordering, negative degree lexicographic ordering, and similar orderings. 
\end{remark}

 $(\mathcal{R}_n, \mathcal{A}_n, \leq_r)$ does not satisfy the definition of SQCC given in Definition \ref{def: sqcc}. In particular, our $\mathcal{H}_n(0)$-action on $\widehat{\mathcal{R}_n}$ required comparisons of leading terms due to the more complicated nature of the $\mathfrak{S}_n$-action on $\mathcal{R}_n$. As a result, we now introduce a weaker compatibility, which is satisfied by all three examples presented.

\begin{definition}\label{def: wqcc}
    Let $M$ be a $\mathfrak{S}_n$-module with basis $\mathcal{B}=\{v_1, \dots, v_n\}$ and a total ordering $\leq$ on $\mathcal{B}$. On the elements of $\mathcal{B}$, let $$\overline{\pi}_i(v_k)=\begin{cases}
            0 & \text{tt}(s_i(v_k))=v_k \\
            -v_k & \text{lt}(s_i(v_k))=v_j> v_k\text{ or } \text{lt}(s_i(v_k))=-v_k \\
            s_i(v_k) & \text{lt}(s_i(v_k))=v_j< v_k
        \end{cases}.$$
     In particular, assume all elements of $\mathcal{B}$ fall into one of these three cases.
    When $\pibar_i$ is a $\mathcal{H}_n(0)$-action on the vector space of $M$, call the resulting module $\widehat{M}$. We say $(M,\mathcal{B},\leq)$ is \textbf{weakly quasisymmetric characteristic compatible} (WQCC) if $$F_{char}^Q(\widehat{M})=F_{char}(M).$$
   
\end{definition}

As the name suggests, every triple $(M, \mathcal{B}, \leq)$ that is SQCC is also WQCC. Additionally, we have:

\begin{theorem}\label{thm: coinv wqcc}
    $(\mathcal{R}_n, \mathcal{A}_n, \leq_r)$ is WQCC.
\end{theorem}
\begin{proof}
    This follows from Theorem \ref{thm: graded coinvariants} and Definition \ref{def: wqcc}. 
\end{proof}

\noindent Extensive computer simulation suggests the following additional example of a WQCC triple:

\begin{conjecture}
    Let $\mathcal{R}_{\mu}$ denote the Garsia-Procesi modules studied in \cite{gpModules} with basis $\mathcal{B}_{\mu}$. Then, under the standard lexicographic monomial ordering $\leq_r$, $(\mathcal{R}_{\mu}, \mathcal{B}_{\mu}, \leq_r)$ is WQCC. 
\end{conjecture}

\section{Open Questions}\label{section: open}

There are two natural motivations for deforming a $\mathfrak{S}_n$-module to create a $\mathcal{H}_n(0)$-module:

\begin{itemize}
    \item To better understand the $\mathfrak{S}_n$-module and
    \item To create natural $\mathcal{H}_n(0)$-modules of interest.
\end{itemize}

Focusing on the first motivation, many of the best known open conjectures in this area of algebraic combinatorics give the graded Frobenius characteristic of an $\mathfrak{S}_n$-module as a conjectured sum expressed in terms of Gessel's fundamental basis and only conjectured to be Schur positive. Such conjectures suggest an important open question:

\begin{oquestion}
    Are there natural conditions one can place on a triple $(M,\mathcal{B},\leq)$ such that the $\pibar_i$ action in Definition \ref{def: wqcc} is always valid and $$F_{char}^Q(\widehat{M})=F_{char}(M)?$$
\end{oquestion}

While this paper explores a uniform deformation of $\mathfrak{S}_n$-modules for the purposes of studying their Frobenius image, our approach has the drawback of mapping quotient spaces to less natural modules from a purely algebraic perspective.  If a primary goal is to learn more about the underlying $\mathfrak{S}_n$-module, this is perhaps a reasonable trade off, but a different sensible goal is to try to preserve quotient spaces as quotient spaces.  A limitation to this approach is that deformations of an $\mathfrak{S}_n$-action as a $\mathcal{H}_n(0)$-action may cause an underlying ideal to no longer be closed under the newly defined action. Thus a a related underlying question is:
\begin{oquestion} 
Let $I$ be an ideal of $\mathfrak{S}_n$-module $M$ which is closed under the $\mathfrak{S}_n$-action.  Is there a natural way to define a related ideal $\widetilde{I}$ and $\mathcal{H}_n(0)$-module $\widetilde{M}$ such that $\widetilde{I}$ is closed under the $\mathcal{H}_n(0)$-action and such that $$F_{char}^Q\left(\faktor{\widetilde{M}}{\widetilde{I}}\right)=F_{char}\left(\faktor{M}{I}\right)?$$
\end{oquestion}

Huang and Rhoades in  \cite{huang2018ordered} explored the answer to this open question in the case of the generalized coinvariant algebra defined by Haglund, Rhoades, and Shimozono in \cite{haglund}. The following definition reduces to $\mathcal{R}_n$, the space we choose to explore earlier in this paper, when restricted to $k=n$. 

\begin{definition}[Huang, Rhoades, Shimozono \cite{haglund}]
    Consider the ideal $$I_{n,k}=\langle x_1^k, \dots, x_n^k, e_n(x_1, \dots, x_n), e_{n-1}(x_1, \dots, x_n), \dots, e_{n-k+1}(x_1, \dots, x_n)\rangle.$$ Define the generalized coinvariant algebra as the quotient $$\mathcal{R}_{n,k}=\faktor{\mathbb{C}[x_1, \dots, x_n]}{I_{n,k}}.$$
\end{definition}
In particular, Huang and Rhoades studied a different deformation of the natural $\mathfrak{S}_n$-action, where the $\mathcal{H}_n(0)$-action is defined by the Demazure operators: $$\pi_i(f)=\frac{x_if-x_{i+1}(s_i(f))}{x_i-x_{i+1}}.$$ However, they noted the ideal $I_{n,k}$ is not closed under this Demazure action. In \cite{huang2018ordered}, they define an analogous ideal $J_{n,k}$ that is closed under the Demazure action. 
\begin{definition}[Huang, Rhoades \cite{huang2018ordered}]
    Define the ideal $$J_{n,k}=\langle e_n(x_1, \dots, x_n), e_{n-1}(x_1, \dots, x_n), \dots, e_{n-k+1}(x_1, \dots, x_n), h_k(x_1), h_k(x_1,x_2),\dots, h_k(x_1, \dots, x_n)\rangle.$$ Then, $$S_{n,k}=\faktor{\mathbb{C}[x_1, \dots, x_n]}{J_{n,k}}.$$
\end{definition}

In particular, their main result is as follows:

\begin{theorem}[Huang, Rhoades \cite{huang2018ordered}]
    $$F_{char}^Q(S_{n,k})=F_{char}(R_{n,k})$$
\end{theorem}
It would be interesting to see if a more general procedure could be developed on all quotient spaces of $\mathbb{C}[x_1,\dots x_n]$ that similarly leave the Frobenius image unchanged and have this construction of Rhoades and Huang as a special case.

\begin{acknowledgement}
We'd like to thank Rob McCloskey and Jonathan Bloom for their thoughtful discussions throughout the development of this paper.  The first author wishes to gratefully acknowledge the Institut Mittag-Leffler for graciously hosting her during a portion of this work. 
\end{acknowledgement}

\bibliographystyle{unsrt}
\bibliography{ref}

\end{document}